\documentclass{amsart}
\RequirePackage[utf8]{inputenc}
\usepackage{amsmath,amsthm,amsfonts,amssymb,amscd,amsbsy,dsfont,multirow,color,pgf, bm}
\usepackage[all]{xy}
\usepackage{graphicx}
\usepackage{hyperref}

\newcommand{\dd}{\mathrm{d}}
\newcommand{\id}{\operatorname{id}}

\newcommand{\Ric}{\operatorname{Ric}}

\newcommand{\R}{\mathds R}

\newcommand{\C}{\mathds C}

\newcommand{\sff}{\mathrm{I\!I}}
\newcommand{\snsq}{\operatorname{sn}_\lambda^2}
\newcommand{\sn}{\operatorname{sn}_\lambda}

\hypersetup{
    pdftoolbar=true,
    pdfmenubar=true,
    pdffitwindow=false,
    pdfstartview={FitH},
    pdftitle={Deformations of free boundary CMC hypersurfaces},
    pdfauthor={Renato G. Bettiol, Paolo Piccione and Bianca Santoro},
    pdfsubject={},
    pdfkeywords={}
    pdfnewwindow=true,
    colorlinks=true, 
    linkcolor=blue,
    citecolor=blue,
    urlcolor=blue,
}

\newtheorem{theorem}{Theorem}[]
\newtheorem{lemma}[theorem]{Lemma}
\newtheorem{proposition}[theorem]{Proposition}
\newtheorem{corollary}[theorem]{Corollary}
\newtheorem*{mainthm}{Main Theorem}
\newtheorem*{theorem*}{Theorem}
\theoremstyle{definition}
\newtheorem{definition}[theorem]{Definition}

\theoremstyle{remark}
\newtheorem{claim}[theorem]{Claim}

\theoremstyle{remark}
\newtheorem{remark}[theorem]{Remark}

\title[Deformations of free boundary CMC hypersurfaces]{Deformations of free boundary CMC hypersurfaces}
\author[R. G. Bettiol]{Renato G. Bettiol}
\author[P. Piccione]{Paolo Piccione}
\author[B. Santoro]{Bianca Santoro}

\address{
\begin{tabular}{lll}
University of Pennsylvania & &Universidade de S\~ao Paulo \\
Department of Mathematics & & Departamento de Matem\'atica \\
209 South 33rd St & & Rua do Mat\~ao, 1010 \\
Philadelphia, PA, 19104-6395, USA & & S\~ao Paulo, SP, 05508-090, Brazil\\
\emph{E-mail address}: {\tt rbettiol@math.upenn.edu} & & \emph{E-mail address}: {\tt piccione@ime.usp.br}\\
\end{tabular}
\bigskip
\hfill\break\hfill\indent
\begin{tabular}{lll}
The City College of New York, CUNY&&\\
Mathematics Department&&\\
138th St and Convent Avenue, NAC 4/112B&&\\
New York, NY, 10031, USA&&\\
\emph{E-mail address}: {\tt bsantoro@ccny.cuny.edu} & &
\end{tabular}
}

\allowdisplaybreaks
\numberwithin{equation}{section}
\numberwithin{theorem}{section}

\thanks{The first named author was supported by the NSF grant DMS-1209387, USA. The second named author is partially supported by Fapesp and CNPq, Brazil. The third named author was partially supported by the NSF grant DMS-1007155 and PSC-CUNY grants, USA}

\subjclass[2010]{53C42, 53A10, 35R35, 58D19, 58E12}

\date{\today}

\begin{document}
\begin{abstract}
We study deformations of free boundary constant mean curvature (CMC) hypersurfaces whose Jacobi operator is degenerate due to symmetries of the ambient space. The value of the mean curvature and the ambient metric are allowed to vary simultaneously, provided that the infinitesimal ambient symmetries change smoothly. We discuss applications to free boundary CMC disks and Delaunay annuli in the unit ball of a space form.
\end{abstract}
\maketitle

\section{Introduction}
Hypersurfaces with constant mean curvature (CMC)
have critical area among those that enclose a fixed volume, and are often used to model interfaces between different media, such as soap films. Free boundary CMC hypersurfaces are particularly important in this regard, as they realize the natural physical constraint of allowing the interface to move freely about the boundary of the considered region.

There are many natural deformation questions regarding a given free boundary CMC hypersurface $\Sigma_0\subset M$ inside a manifold with boundary. For instance, consider the problems of deforming $\Sigma_0$ through {\em other free boundary} hypersurfaces in $M$ varying the value of the mean curvature, or {\em deforming the ambient metric} on $M$ while retaining the existence of a free boundary CMC hypersurface. Any answer to these deformation questions carries an intrinsic ambiguity due to ambient isometries in $M$, and should hence be considered up to such symmetries. In the absence of boundary conditions, these problems have been addressed in \cite{BetPicSic2014,BetPicSic2010}. The main result of the present paper is a simultaneous answer to both deformation questions above for \emph{free boundary} CMC hypersurfaces. The key assumption to ensure that these deformations are feasible is that $\Sigma_0$ is \emph{equivariantly nondegenerate}, meaning that the only Jacobi fields along $\Sigma_0$ originate from ambient symmetries (see Definition~\ref{def:equivnondeg}). Using this hypothesis, we have the following deformation result:

\begin{mainthm}
Let $\Sigma_0$ be an equivariantly nondegenerate free boundary CMC hypersurface of $(M,g_{\lambda_0})$ with mean curvature $h_0$. Assume that $g_\lambda$ is a smooth family of Riemannian metrics on $M$ whose Killing fields vary smoothly with $\lambda$.
Then, $\Sigma_0$ is the member $\Sigma_{(h_0,\lambda_0)}$ of a smooth family $\Sigma_{(h,\lambda)}$ of free boundary CMC hypersurfaces in $(M,g_\lambda)$ with mean curvature $h$, which is locally unique up to ambient isometries.
\end{mainthm}

Further technical details on the above statement are discussed in Theorem~\ref{thm:main}. The core of the proof is a convenient formulation of the prescribed mean curvature problem, that allows to combine the Implicit Function Theorem and an argument using the flux of Killing fields to obtain the desired equivariant deformation, see \eqref{eq:deftildeH}. This is inspired by an idea of Kapouleas~\cite{Kap1,Kap2}, see also \cite{BetPicSic2010}, which we adapt to the free boundary framework with varying ambient metrics.

In certain cases, up to using ambient isometries, the family $\Sigma_{(h,\lambda_0)}$ determines a foliation by free boundary CMC hypersurfaces near $\Sigma_{(h_0,\lambda_0)}$. Sufficient conditions for this are given in Propositions~\ref{prop:Lpsi=1} and \ref{prop:foliatedeqnondeg}. For instance, spherical caps are free boundary CMC hypersurfaces that clearly foliate the unit ball $B^{n+1}_\lambda$ in a space form $M^{n+1}_\lambda$, see Figure~\ref{fig:sphcaps}. This fact is reobtained as a consequence of our main result applied to a flat disk $D^n\subset B^{n+1}_0$, which is equivariantly nondegenerate.

A more complex situation arises by considering CMC annuli instead of disks, which can be seen as a deformation of the so-called \emph{critical catenoid}. This is the portion of a catenoid that meets the boundary of the unit ball $B^3_0$ orthogonally, and has been the object of several recent investigations by Fraser and Schoen~\cite{fraser-schoen2,fraser-schoen3,fraser-schoen1} and others \cite{devyver,smith-zhou,tran}. We establish the equivariant nondegeneracy of the critical catenoid (Lemma~\ref{lemma:catenoidnondegenerate}), allowing us to deform it through other free boundary CMC annuli in the unit ball of a space form. These surfaces can be recognized as compact portions of Delaunay surfaces, see Section~\ref{sec:delaunay}. In the particular case of fixing the Euclidean ambient metric and only varying the mean curvature, this family of Delaunay annuli containing the critical catenoid is formed by portions of unduloids and nodoids. These are rotationally symmetric surfaces that intersect the unit ball orthogonally and whose profile curve is the roulette of a conic section, that is, the path traced by one of the foci of a conic as it rolls without slipping along a straight line (see Figure~\ref{fig:roulette}). The critical catenoid corresponds to rolling a parabola, which is the limiting case of rolling ellipses or hyperbolas with eccentricity close to $1$, that respectively give rise to the aforementioned unduloids and nodoids.
For potential applications to other surfaces with more complicated topology, see Remark~\ref{rem:moretopology}.

This paper is organized as follows. The basic framework for the free boundary CMC problem is recalled in Section~\ref{sec:basic}, and preliminary computations regarding the flux of Killing fields are made in Section~\ref{sec:killingflux}. In Section~\ref{sec:normalhypersurfaces}, we discuss the structure of the set of hypersurfaces with fixed diffeomorphism type in a manifold with boundary $M$ that meet the boundary $\partial M$ orthogonally. The detailed statement and proof of our main deformation result (Theorem~\ref{thm:main}) is given in Section~\ref{sec:TFI-CMC}.
Issues regarding foliations by free boundary CMC hypersurfaces are discussed in Section~\ref{sec:foliation}. Finally, Sections~\ref{sec:disks} and \ref{sec:delaunay} contain applications to free boundary CMC disks and annuli in the unit ball of a space form, respectively.

\smallskip
\noindent
{\bf Acknowledgements.} It is a pleasure to thank Rafe Mazzeo for many valuable suggestions, including the key ingredient for the proof of Proposition~\ref{prop:Lpsi=1}.

\section{Free boundary CMC hypersurfaces}\label{sec:basic}

\subsection{Basic definitions}
Let $(M,g)$ be an $(n+1)$-dimensional Riemannian manifold with boundary.
We denote by $\vec n_{\partial M}$ the outer unit normal vector field to the boundary of $M$. This choice is used to define the \emph{second fundamental form}
\begin{equation*}
\sff^{\partial M}(X,Y):=g\big(\nabla_X Y,\vec n_{\partial M}\big).
\end{equation*}
The main objects studied in this paper are compact hypersurfaces $\Sigma$ of $M$, whose boundary
$\partial\Sigma$ lies in $\partial M$.
In order to simplify notation for these objects, we identify embeddings $x\colon\Sigma\hookrightarrow M$ with their image $x(\Sigma)\subset M$.
\begin{definition}\label{def:normalhypersurfaces}
We say that a hypersurface $\Sigma\subset M$ is \emph{admissible} if
\begin{itemize}
\item[(a)] $\Sigma\cap\partial M=\partial\Sigma$;
\item[(b)] the normal bundle $T\Sigma^\perp$ is orientable.
\end{itemize}
If, in addition, $\vec n_{\partial M}(p)\in T_p\Sigma$ for all $p\in\partial\Sigma$, we say $\Sigma$ is \emph{normal}.
An admissible hypersurface $\Sigma$ is said to \emph{bound a finite volume} if
\begin{itemize}
\item[(c)] $M\setminus\Sigma=\Omega_1\cup\Omega_2$, with $\overline\Omega_1$ compact and $\Omega_1\cap\Omega_2=\emptyset$.
\end{itemize}
\end{definition}
\begin{remark}\label{rem:orthrelation}
Note that if $\Sigma$ is normal, then $\Sigma$ and $\partial M$ are transverse submanifolds. Moreover, for all $p\in\partial\Sigma$, $\vec n_{\partial M}(p)\in T_p(\partial\Sigma)^\perp\cap T_p\Sigma$.
\end{remark}

\begin{remark}\label{rem:stability}
Observe also that the properties (a), (b) and (c) in Definition~\ref{def:normalhypersurfaces} are stable under $C^1$-perturbations of $\Sigma$.
\end{remark}

Given an admissible hypersurface $\Sigma$, we choose an orientation for $T\Sigma^\perp$, and denote by $\vec n_\Sigma$ the positively oriented unit normal vector field along $\Sigma$. When $\Sigma$ bounds a finite volume, the positive orientation of $T\Sigma^\perp$ is chosen to point towards the bounded region $\Omega_1$.  The \emph{second fundamental form of $\Sigma$} is given by
\begin{equation*}
\sff^{\Sigma}(X,Y):=g\big(\nabla_X Y,\vec n_{\Sigma}\big),
\end{equation*}
and its trace is called the \emph{mean curvature function} $H_\Sigma$. The \emph{mean curvature vector} of $\Sigma$ is defined as $\vec H_\Sigma:=H_\Sigma \,\vec n_\Sigma$. If $H_\Sigma$ is constant, then $\Sigma$ is called a \emph{constant mean curvature hypersurface} (or \emph{CMC hypersurface}), and a \emph{minimal hypersurface} if the value of this constant is zero.

\subsection{Regularity assumptions}
The central idea of the present paper is to consider \emph{perturbations} of a given free boundary CMC hypersurface $\Sigma_0\subset M$ by nearby admissible normal hypersurfaces. Note that the CMC condition implies that $\Sigma_0$ is necessarily smooth. A technical, nevertheless essential, point is to establish an appropriate notion of regularity for perturbed hypersurfaces and, related to this choice, a topology on the set of hypersurfaces in $M$. As discussed in Section~\ref{sec:normalhypersurfaces}, the set of perturbations of $\Sigma_0$ having a given regularity can be parametrized, via the normal exponential map, by real-valued functions on $\Sigma_0$ with the same regularity.

There are two issues related to the choice of regularity: the question of Fredholmness for the Jacobi operator (see Subsection~\ref{sub:Jacobifields}), and, more specifically for the present work, the question of existence of a bounded right-inverse for the map that carries a function on a compact manifold to its normal derivative along the boundary (see Section~\ref{sec:normalhypersurfaces}).

A customary choice when studying variations of a CMC hypersurface is to require a H\"older-type regularity $C^{2,\alpha}$, for some $\alpha\in(0,1)$. This provides the correct framework for the Fredholmness of the Jacobi operator, see below, as well as for the right-inverse problem, via a Whitney-type extension theorem (see \cite[Chap.\ VI]{stein}). We use $C^{2,\alpha}$-regularity assumptions throughout, which puts us in the framework of (non-separable) \emph{Banach} manifolds.

An alternative (separable) \emph{Hilbert} manifold approach could be obtained by requiring a Sobolev-type $H^s$-regularity; using standard Sobolev embedding theorems
(see for instance \cite[Ch.\ 2, \S~9, Thm.\ 2.30, p.\ 50]{Aubin_book}), in order to have inclusion in $C^2$, one needs to assume that $s>2+\tfrac12\mathrm{dim}(\Sigma)=2+\frac n2$. Also in this case, Fredholmness for the Jacobi operator is obtained readily from the Fredholmness of the Laplacian, and the right-inverse question is answered by the Sobolev Trace Theorem for the normal derivative (see \cite[Thm.\ 8.3, p.\ 39]{LionsMagenes}).

\subsection{Jacobi fields}\label{sub:Jacobifields}
Given a free boundary CMC hypersurface $x_0\colon\Sigma\hookrightarrow M$, denote by $J_{x_0}$ its \emph{Jacobi operator}, which is the second-order linear differential operator
\begin{equation}\label{eq:Jacobiop}
J_{x_0}(\psi):=\Delta_{\Sigma}\psi-\big(\|\sff^{\Sigma}\|^2+\Ric_g(\vec n_\Sigma,\vec n_\Sigma)\big)\psi,
\end{equation}
defined on the space of $C^2$ functions $\psi\colon\Sigma\to\R$. In the above, $\Delta_{\Sigma}$ is the Laplacian of $\big(\Sigma,x_0^*(g)\big)$ with nonnegative spectrum, and $\|\sff^{\Sigma}\|$ is the Hilbert-Schmidt norm of the second fundamental form of $x_0\colon\Sigma\hookrightarrow M$.
A \emph{Jacobi field} along $x_0$ is a smooth function $\psi\colon\Sigma\to\R$ satisfying $J_{x_0}(\psi)=0$.

Let $x_r\colon\Sigma\hookrightarrow M$, $r\in(-\varepsilon,\varepsilon)$, be a smooth variation of $x_0$, where each $x_r$ is a CMC embedding with mean curvature $H_r$. Let $V=\frac{\mathrm d}{\mathrm dr}\big\vert_{r=0}x_r$ be the corresponding variational vector field. Then $\psi_V:=g\big(V,\vec n_\Sigma)$ satisfies
\begin{equation}\label{eq:Jacobionfunctions}
J_{x_0}(\psi_V)=\tfrac{\dd}{\dd r}\big\vert_{r=0}H_r,
\end{equation}
so that $\psi_V$ is a Jacobi field exactly when $H_r\equiv H_0$ to first order as $r\to0$.
Thus, $J_{x_0}$ is interpreted as the
linearization of the mean curvature $H_\Sigma$ at $x_0(\Sigma)$.
Furthermore, if each $x_r$ is a free boundary CMC hypersurface, then $\psi=\psi_V$ satisfies the \emph{linearized free boundary condition}
\begin{equation}\label{eq:linearizedfreebdy}
g\big(\nabla \psi,\vec n_{\partial M}\big)+\sff^{\partial M}\big(\vec n_\Sigma,\vec n_\Sigma\big)\,\psi=0.
\end{equation}

The restriction of $J_{x_0}$ to the closed subspace
\begin{equation}\label{eqn:BC}
C^{2,\alpha}_\partial(\Sigma):=\big\{\psi\in C^{2,\alpha}(\Sigma):g\big(\nabla \psi,\vec n_{\partial M}\big)+\sff^{\partial M}\big(\vec n_\Sigma,\vec n_\Sigma\big)\,\psi=0
\big\}
\end{equation}
is a Fredholm operator of index zero that takes values in $C^{0,\alpha}(\Sigma)$, see \cite[Sec~2]{MaxNunSmi13}. The free boundary CMC hypersurface $x_0\colon\Sigma\hookrightarrow M$ is called \emph{nondegenerate} if $\ker J_{x_0}\cap C^{2,\alpha}_\partial(\Sigma)=\{0\}$, that is, if $J_{x_0}\colon C^{2,\alpha}_\partial(\Sigma)\to C^{0,\alpha}(\Sigma)$ is an isomorphism. The dimension of $\ker J_{x_0}\cap C^{2,\alpha}_\partial(\Sigma)$ is called the \emph{nullity} of $x_0$. We warn the reader that other conventions for the nullity of $x_0$ can be found in the literature; for instance, it is sometimes defined as the dimension of $\ker J_{x_0}$ disregarding \eqref{eq:linearizedfreebdy}.

\subsection{Killing-Jacobi fields}
We say that a (local) vector field $K$ on $M$ is a \emph{(local) Killing vector field} if $K_p\in T_p(\partial M)$ for all $p\in \partial M$ in the domain of $K$, and $\mathcal L_K g=0$.

Given a free boundary CMC hypersurface $x_0\colon\Sigma\hookrightarrow M$, every local Killing vector field $K$ whose domain contains $x_0(\Sigma)$ induces a Jacobi field
\begin{equation}\label{eq:defKillJac}
\psi_K:=g(K,\vec n_\Sigma)
\end{equation}
along $x_0$, with $\psi_K\in C^{2,\alpha}_\partial(\Sigma)$. Indeed, denote by $I_s$, $s\in(-\varepsilon,\varepsilon)$, the local flow of $K$. Since each $I_s$ is a local isometry that preserves $\partial M$, the embeddings $x_s=I_s\circ x_0$ are free boundary CMC hypersurfaces with the same mean curvature as $x_0$. Recalling formula \eqref{eq:Jacobionfunctions}, this implies that $\psi_K$ is a Jacobi field in $C^{2,\alpha}_\partial(\Sigma)$. Jacobi fields of the form \eqref{eq:defKillJac} are called \emph{Killing-Jacobi fields}.

\begin{definition}\label{def:equivnondeg}
The free boundary CMC hypersurface $x_0\colon\Sigma\hookrightarrow M$ is called \emph{equivariantly nondegenerate} if $\ker J_{x_0}\cap C^{2,\alpha}_\partial(\Sigma)$ consists only of Killing-Jacobi fields.
\end{definition}

Given a free boundary CMC hypersurface $x\colon\Sigma\to M$ and a local isometry $I$ of $(M,g)$ whose domain contains $x(\Sigma)$, then $I\circ x$ is another free boundary CMC hypersurface. Local isometries of $(M,g)$ preserve $\partial M$, and therefore also the orthogonality condition. Free boundary CMC embeddings $x,x'\colon\Sigma\hookrightarrow M$ are said to be \emph{congruent} if there exists a local isometry $I$ of $(M,g)$ such that $I\big(x(\Sigma)\big)=x'(\Sigma)$.

\section{Flux of Killing fields}\label{sec:killingflux}

In preparation for the proof of our main result, we study the flux of Killing vector fields across hypersurfaces. This provides a convenient approach to study the equation $H_\Sigma\equiv\mathrm{const.}$ modulo ambient isometries, which originates from the work of Kapouleas~\cite{Kap1,Kap2}, see also \cite{BetPicSic2010}.

\begin{proposition}\label{thm:integralidentities}
Let $\Sigma\subset M$ be an admissible hypersurface and $K$ be a Killing vector field on $M$.
\begin{itemize}
\item[(A)] If $\Sigma$ is normal, then
\begin{equation}\label{eq:secondidentity}
\int_\Sigma g(K,\vec H_\Sigma)=0.
\end{equation}
\item[(B)] If $\Sigma$ bounds a finite volume, then
\begin{equation}\label{eq:firstidentity}
\int_\Sigma g(K,\vec n_\Sigma)=0.
\end{equation}
\end{itemize}
\end{proposition}

\begin{proof}
For (A), in order to prove \eqref{eq:secondidentity}, denote by $K_\Sigma$ the component of $K$ tangent to $\Sigma$.  An easy computation gives $\mathrm{div}_\Sigma(K_\Sigma)=g(K,\vec H_\Sigma)$. Moreover, since $\Sigma$ is normal, $K_\Sigma$ is tangent to $\partial\Sigma$ at every point $p\in\partial\Sigma$ (see Remark~\ref{rem:orthrelation}). Hence, $\int_\Sigma g(K,\vec H_\Sigma)=\int_\Sigma\mathrm{div}_\Sigma(K_\Sigma)=0$,
by Stokes' theorem.

Let us now prove (B).
Since $\partial\Omega_1=\Sigma\cup(\partial\Omega_1\cap\partial M)$, and $K$ is tangent to $\partial M$,
we have that
\begin{equation*}
\int_{\Sigma}g(K,\vec n_\Sigma) = \int_{\partial\Omega_1} g(K,\vec n_{\partial\Omega_1})=\int_{\Omega_1}\mathrm{div}_g(K)=0,
\end{equation*}
where the second equality follows from Stokes' theorem, and the third from the fact that Killing vector fields have zero divergence.
\end{proof}

\begin{corollary}\label{thm:Htildeconstant}
Let $\Sigma$ be a normal hypersurface of $M$, and let $K_1,\ldots,K_r$ be Killing vector fields in $M$ such that the corresponding Killing-Jacobi fields $\psi_i\colon \Sigma\to\R$,
\begin{equation*}
\psi_i=g(K_i,\vec n_\Sigma), \quad i=1,\ldots,r,
\end{equation*}
are linearly independent. Let $h$ be a continuous function on $\Sigma$ which is $L^2$-orthogonal to the space generated by the $\psi_i$, i.e., $\int_\Sigma h\psi_i=0$ for all $i$.

If $\alpha_1,\ldots,\alpha_r\in\mathds R$ are such that $H_\Sigma+\sum_{i=1}^r\alpha_i \psi_i=h$, then necessarily $\alpha_i=0$ for all $i$ (and thus $H_\Sigma=h$).

If $\Sigma$ bounds a finite volume, then the conclusion above holds for all constant functions $h$, i.e.,  the function $H_\Sigma+\sum_{i=1}^r\alpha_i \psi_i$ is constant if and only if $H_\Sigma$ is constant and $\alpha_i=0$ for all $i$.
\end{corollary}

\begin{proof}
Suppose $H_\Sigma+\sum_{i=1}^r\alpha_i \psi_i\equiv h$. Multiplying both sides by $\sum\limits_{i=1}^r\alpha_i \psi_i$ and integrating over $\Sigma$, we have
\begin{eqnarray*}
\int_\Sigma\left(\sum_{i=1}^r\alpha_i \psi_i\right)^2&=&\sum_{i=1}^r \alpha_i\int_\Sigma h\psi_i-\sum_{i=1}^r\alpha_i\int_\Sigma H_\Sigma\psi_i=
-\sum_{i=1}^r\alpha_i\int_\Sigma H_\Sigma\psi_i\\ &=&-\sum_{i=1}^r\alpha_i\left(\int_\Sigma g(K_i,\vec H_\Sigma)\right)\stackrel{\eqref{eq:secondidentity}}=0,
\end{eqnarray*}
Thus, $\sum_{i=1}^r\alpha_i \psi_i=0$. Since the $\psi_i$ are linearly independent, it follows that $\alpha_i=0$, $i=1,\ldots,r$, and $H_\Sigma=h$, which proves the first statement.

Under the assumption that $\Sigma$ bounds a finite volume, identity \eqref{eq:firstidentity} implies that any constant function is $L^2$-orthogonal to the space generated by the $\psi_i$, and the conclusion above follows.
\end{proof}

\begin{remark}
Statement (B) of Proposition~\ref{thm:integralidentities} and the second statement of Corollary~\ref{thm:Htildeconstant} hold true for all local Killing fields $K$ defined on a neighborhood of $\Omega_1$, while in statement (A) of Proposition~\ref{thm:integralidentities} and the first statement of Corollary~\ref{thm:Htildeconstant}, it suffices to assume that the Killing fields are defined in a neighborhood of $\Sigma$.
\end{remark}

\section{The manifold of normal hypersurfaces}\label{sec:normalhypersurfaces}

In this section, we give details on the existence of a natural smooth structure on the set of normal hypersurfaces of a Riemannian manifold $(M,g)$ that are \emph{close} to a given compact normal hypersurface $\Sigma$ in some suitable topology.
As usual, perturbations of $\Sigma$ are parametrized as  graphs over $\Sigma$ of sufficiently small smooth functions $f\colon\Sigma\to\R$. The property of being a normal hypersurface (see Definition~\ref{def:normalhypersurfaces}) may not be preserved under the normal exponential flow of $\Sigma$, unless $\partial M$ is totally geodesic in $M$.
However, up to using the exponential flow of an auxiliary metric for which $\partial M$ is totally geodesic and has the same normal bundle, we may assume without loss of generality for the remainder of this section that $\partial M$ is totally geodesic in $(M,g)$.

A detailed description of the smooth structure on the set of submanifolds with fixed diffeomorphism type can be found in the series of papers by Michor \cite{mich1, mich2, mich3}, in the $C^\infty$ case. In order to fulfill the appropriate technical requirements of the Implicit Function Theorem,
one has to go beyond the $C^\infty$ realm, using embeddings of H\"older class $C^{2,\alpha}$.
There are several issues concerning the regularity of the set of unparametrized embeddings in this low regularity
setting, which are extensively discussed in \cite{AliPic}.
In this paper, however, we are only interested in the smooth structure near a given smooth unparametrized
embedding, which avoids all the subtleties involved in the lack of regularity of the change of coordinates.

\subsection{Unparametrized normal embeddings}
The appropriate setup for studying the set of submanifolds of a given diffeomorphism type is obtained by considering the notion of \emph{unparametrized embeddings}.
 Given a compact manifold $\Sigma$, we say that two embeddings $x_1,x_2\colon\Sigma\hookrightarrow M$ are {\em equivalent} if there exists a diffeomorphism $\phi\colon\Sigma\to\Sigma$ such that $x_1=x_2\circ\phi$. Equivalence classes of embeddings
 are called \emph{unparametrized embeddings of $\Sigma$ in $M$}.

We denote by $\mathcal E(\Sigma,M)$ the space of $C^{2,\alpha}$-un\-parametrized embeddings of $\Sigma$ in $M$, and by $\mathcal E_\partial(\Sigma,M)$ the subset of $\mathcal E(\Sigma,M)$ consisting of unparametrized embeddings $x\colon\Sigma\hookrightarrow M$ such that $x(\partial\Sigma)\subset\partial M$. Given an embedding $x\colon\Sigma\hookrightarrow M$, we denote by $[x]\in\mathcal E(\Sigma,M)$ the unparametrized embedding defined by $x$. Note that $[x]$ is uniquely determined by the image $x(\Sigma)$. The notions of \emph{admissible} and \emph{normal} hypersurfaces (Definition~\ref{def:normalhypersurfaces})
 extend naturally to unparametrized embeddings. Let $\mathcal E_\partial^\perp(\Sigma,M)$ denote the subset of
 $\mathcal E_\partial(\Sigma,M)$ consisting of unparametrized {\em normal} embeddings. Observe that the admissible normal embeddings form an open subset of $\mathcal E_\partial^\perp(\Sigma,M)$, see Remark~\ref{rem:stability}.

\begin{proposition}\label{prop:normhyperman}
Let $\Sigma$ be a compact manifold with boundary and $x_0\colon\Sigma\hookrightarrow M$ be an admissible smooth
normal embedding. A sufficiently small neighborhood of $[x_0]$ in $\mathcal E_\partial^\perp(\Sigma,M)$ can be
identified with an infinite-dimensional smooth submanifold $\mathcal N$ of the Banach space
$C^{2,\alpha}(\Sigma)$, with $0\in\mathcal N$ corresponding to $[x_0]$, such that
$T_0\mathcal N=C^{2,\alpha}_\partial(\Sigma)$, see \eqref{eqn:BC}.

Moreover, assume that $K_1,\ldots,K_r$ is a family of local Killing vector fields defined in a neighborhood of
$x_0(\Sigma)$, and consider the functions $\psi_i=g(K_i,\vec n_0)$ in $C^{2,\alpha}_\partial(\Sigma)$.
Then, the pseudo-group of local isometries generated by the $K_i$ has a continuous local action on $\mathcal N$,
and the orbit of $0$ under this action is a smooth submanifold of $\mathcal N$, of dimension greater than or equal
to the dimension of the span of the $\psi_i$.
\end{proposition}

\begin{proof}
To each sufficiently small $C^{2,\alpha}$-map $f\colon\Sigma\to\R$, we associate the embedding
\begin{equation}\label{eq:xf}
x_f\colon\Sigma\hookrightarrow M, \quad x_f(p):=\exp_{x_0(p)}\big(f(p)\,\vec n_0\big),
\end{equation}
where $\vec n_0$ is the unit normal vector field along $x_0\colon\Sigma\hookrightarrow M$. Since $x_0$ is
a normal embedding, $\vec n_0(p)\in T_p(\partial M)$ for all $p\in x_0(\Sigma)\cap\partial M$, and hence
$x_f(\partial\Sigma)\subset\partial M$, because $\partial M$ is totally geodesic.
Clearly, if $f\equiv0$, then $x_f=x_0$.  In order to simplify notation, we denote the image of
$x_f\colon\Sigma\hookrightarrow M$ by $\Sigma_f$; in particular, $\Sigma_0$ is the image of the original
embedding $x_0\colon\Sigma\hookrightarrow M$.

The above correspondence $f\mapsto \Sigma_f$ gives a bijection from a sufficiently small neighborhood
$\mathcal U$ of $0\in C^{2,\alpha}(\Sigma)$ to a neighborhood $\mathcal V$ of $[x_0]\in\mathcal E_\partial(\Sigma,M)$.
Denote by  $\mathcal N$  the subset of $\mathcal U$ consisting of maps $f$ such that $x_f$ is a normal embedding.

\begin{claim}
$\mathcal N$ is a submanifold of $C^{2,\alpha}(\Sigma)$.
\end{claim}
In order to prove this claim,  let $\varepsilon>0$ be small and consider the diffeomorphism
\begin{equation*}
\Phi\colon\Sigma\times[0,\varepsilon)\longrightarrow\mathcal A,\quad \Phi(p,t):=\exp_{x_0(p)}\big(t\,\vec n_0(p)\big),
\end{equation*}
where $\mathcal A$ is an open subset of $M$ containing $\Sigma_0$.
Note that $\Phi(\Sigma,0)=x_0(\Sigma)$ is the original normal hypersurface $\Sigma_0$, and that $\dd\Phi_{(p,t)}$ is an isomorphism for $t>0$ sufficiently small. In fact,
\begin{equation}\label{eq:dPhi0p}
\dd\Phi_{(p,0)}(v,\tau)=\tau\,n_0(p)+v,
\end{equation}
for all $\tau\in\mathds R$ and $v\in T_p\Sigma_0$. For $(p,t)\in\partial\Sigma\times[0,\varepsilon)$, set
\begin{equation}\label{eq:dpsin}
\big(X(p,t),\zeta(p,t)\big):=\dd\Phi_{(p,t)}^{-1}\big(\vec n_{\partial M}(\Phi(p,t))\big)\in T_p\Sigma\times\R;
\end{equation}
note that $\zeta$ is a smooth function on $\partial\Sigma\times\left[0,\varepsilon\right)$, and
$X$ is a smooth time-dependent vector field along $\partial\Sigma$ which is tangent to $\Sigma$.

For $f\in\mathcal U$, $x_f(p)=\Phi\big(p,f(p)\big)$, and so $T_{x_f(p)}\Sigma_f=\mathrm{Im}\left[\mathrm d\Phi_{(p,f(p))}\circ\big(\mathrm{Id}, \mathrm df(p)\big)\right]$.
Thus,
\begin{equation}\label{eq:dpsitang}
\dd\Phi_{(p,f(p))}^{-1}\left[T_{x_f(p)} \Sigma_f \right]=\mathrm{Im}\big(\mathrm{Id}, \mathrm df(p)\big)=\mathrm{Graph}\big(\dd f(p)\big).
\end{equation}
The embedding $x_f$ is normal if and only if $\vec n_{\partial M}\big(x_f(p)\big)\in T_{x_f(p)}\Sigma_f$ for
all $p\in\partial\Sigma$. From \eqref{eq:dpsin} and \eqref{eq:dpsitang}, this condition reads
\[\phantom{,\quad\mbox{ for all } p\in\partial\Sigma}\zeta\big(p,f(p)\big)=\dd f(p)\cdot X\big(p,f(p)\big)\quad\mbox{ for all } p\in\partial\Sigma.\]
In other words, the set of normal embeddings $x_f$ is identified with the inverse image $\eta^{-1}(0)$, where
$\eta$ is the smooth map $C^{2,\alpha}(\Sigma)\ni f\mapsto\eta(f)\in C^{1,\alpha}(\partial\Sigma)$ defined by
\[\phantom{,\quad p\in\partial\Sigma.}\eta(f)(p)=\mathrm df(p)\cdot X\big(p,f(p)\big)-
\zeta\big(p,f(p)\big),\quad p\in\partial\Sigma.\]
We prove that $\mathcal N$ is a smooth submanifold of $C^{2,\alpha}(\Sigma)$ by showing that $\eta$ is
a submersion at $f=0$. The linearization of $\eta$ at $0$ applied to $\psi\in C^{2,\alpha}(\Sigma)$ reads
\begin{equation}\label{eq:dtheta0}
\dd\eta_0(\psi)(p)=\mathrm d\psi(p)\cdot X(p,0)-\frac{\partial\zeta}{\partial t}(p,0)\,\psi(p),\quad p\in\partial\Sigma.
\end{equation}
Note that $X(p,0)=\vec n_{\partial M}(p)$, which gives a unit orthogonal field along $\partial\Sigma$
(see Remark~\ref{rem:orthrelation}).
\begin{claim}
\label{claim:sub}
For all $p\in\partial\Sigma$,
\begin{equation}\label{eq:partiallambda}
\frac{\partial\zeta}{\partial t}(p,0)=-\sff^{\partial M}\big(\vec n_0(p),\vec n_0(p)\big).
\end{equation}
\end{claim}

Let us first assume the above claim and complete the proof of the theorem.

Proving that $\eta\colon C^{2,\alpha}(\Sigma)\to C^{1,\alpha}(\partial\Sigma)$ is a submersion\footnote{If one considers the Sobolev regularity $H^s$, $s>2+\frac n2$, then the natural range of the map $\eta$ is $H^{s-3/2}(\partial\Sigma)$, and the existence of a bounded right-inverse for $\dd\eta_0$ is given by the Sobolev Trace Theorem, see for instance \cite[Thm.~8.3, p.\ 39]{LionsMagenes}.}
at $f=0$ amounts to showing that $\dd\eta_0$ is surjective and that $\ker \dd\eta_0$ is complemented in $C^{2,\alpha}(\Sigma)$. This is equivalent to the existence of a bounded linear right-inverse for the map $\dd\eta_0$. Such a right-inverse is given in Lemma~\ref{thm:rightinverse} below.

Finally, $T_0\mathcal N=C^{2,\alpha}_\partial(\Sigma)$ follows from $T_0\mathcal N=\ker\dd\eta_0$,
using \eqref{eq:dtheta0}, \eqref{eq:partiallambda}, and the identity \eqref{eq:in0ep} below.

As to the natural action of the (pseudo-)group of isometries generated by a family of (local) Killing fields on the set of $C^{2,\alpha}$-unparametrized embeddings, note that this restricts to an action on the set of unparametrized normal embeddings of $\Sigma$ in $M$, i.e., a local action on $\mathcal N$. This action is only continuous, but the orbit of any smooth map is a smooth submanifold \cite[Prop.\ 5.1]{AliPic}. It is easy to see that the tangent space to the orbit through $0$ contains the functions $\psi_i=g(K_i,\vec n_0)$.
\end{proof}

\begin{proof}[Proof of Claim~\ref{claim:sub}]
Rewrite equation \eqref{eq:dpsin} as
\begin{equation}\label{eq:star}
\dd\Phi_{(p,t)}\big(X(p,t),\zeta(p,t)\big)=\vec n_{\partial M}\big(\Phi(p,t)\big).
\end{equation}
Using \eqref{eq:dPhi0p}, \eqref{eq:star}, and that $\vec n_{\partial M}(p)\in T_p\Sigma_0$ for $p\in\partial\Sigma$, we have:
\begin{equation}\label{eq:in0ep}
\zeta(p,0)=0,\quad\text{and}\quad X(p,0)=\vec n_{\partial M}(p).
\end{equation}
Fix $p\in\partial\Sigma$, and consider \eqref{eq:star} as an equality between vector fields along the geodesic $t\mapsto\Phi(p,t)$. Let us differentiate covariantly (using the Levi-Civita connection of $g$)
identity \eqref{eq:star} at $t=0$.
For $s,t\in\mathds R$ small, set:
\[\rho(t,s):=\mathrm d\Phi_{(p,t)}\big(X(p,t),\zeta(p,t)\big),\]
so that the left-hand side of \eqref{eq:star} is given by $\rho(t,t)$.
Denoting covariant derivatives along curves by $\frac{\mathrm D}{\mathrm dt}$ and
$\frac{\mathrm D}{\mathrm ds}$, we have:
\begin{equation}\label{eq:clearly}
\tfrac{\mathrm D}{\mathrm dt}\big\vert_{t=0}\rho(t,t)=\tfrac{\mathrm D}{\mathrm dt}\big\vert_{t=0}\rho(t,0)+\tfrac{\mathrm D}{\mathrm ds}\big\vert_{s=0}\rho(0,s).
\end{equation}
Moreover,
\[\rho(t,0)\stackrel{\eqref{eq:in0ep}}=\mathrm d\Phi_{(p,t)}\big(\vec n_{\partial M}(p),0\big),\qquad
\rho(0,s)\stackrel{\eqref{eq:dPhi0p}}=\zeta(p,s)\,\vec n_0(p)+X(p,s).\]
In order to differentiate $\rho(t,0)$, let us consider a smooth curve $u\mapsto p(u)\in\Sigma_0$ with $p(0)=p$ and $p'(0)=\vec n_{\partial M}(p)$, so that $\rho(t,0)=\tfrac{\mathrm d}{\mathrm du}\big\vert_{u=0}\Phi\big(p(u),t\big)$. Then
\begin{multline}\label{eq:derivatarhot0}
\tfrac{\mathrm D}{\mathrm dt}\big\vert_{t=0}\rho(t,0)=\tfrac{\mathrm D}{\mathrm dt}\big\vert_{t=0}
\tfrac{\mathrm d}{\mathrm du}\big\vert_{u=0}\Phi\big(p(u),t\big)\\=
\tfrac{\mathrm D}{\mathrm du}\big\vert_{u=0}
\tfrac{\mathrm d}{\mathrm dt}\big\vert_{t=0}\Phi\big(p(u),t\big)=\tfrac{\mathrm D}{\mathrm du}\big\vert_{u=0}\vec n_\Sigma\big(p(u)\big)=\nabla_{\vec n_{\partial M}(p)}\vec n_0.
\end{multline}
The derivative of $\rho(0,s)$ is given by
\begin{equation}\label{eq:derivatarho0s}
\tfrac{\mathrm D}{\mathrm ds}\big\vert_{s=0}\rho(0,s)=\frac{\partial\zeta}{\partial t}(p,0)\,\vec n_0(p)+\frac{\partial X}{\partial t}(p,0).
\end{equation}
Note that, for $p$ fixed, $t\mapsto X(p,t)$ is a curve in the fixed vector space $T_p\Sigma_0$, and
$\frac{\partial X}{\partial t}(p,0)\in T_p\Sigma_0$ is the standard derivative of this curve at $t=0$.
Finally, the derivative of the right-hand side of \eqref{eq:star} is given by
\begin{equation}\label{eq:derlatodestro}
\tfrac{\mathrm D}{\mathrm dt}\big\vert_{t=0}\vec n_{\partial M}\big(\Phi(p,t)\big)=\nabla_{\vec n_0(p)}\vec n_{\partial M}.
\end{equation}
Using \eqref{eq:clearly}, \eqref{eq:derivatarhot0} and \eqref{eq:derivatarho0s}, we obtain that the covariant derivative of equation \eqref{eq:star} at $t=0$ reads (cf.\ \cite[p.\ 174]{schoen}):
\begin{equation*}
\nabla_{\vec n_{\partial M}(p)}\vec n_\Sigma+\frac{\partial\zeta}{\partial t}(p,0)\,\vec n_\Sigma(p)+\frac{\partial X}{\partial t}(p,0)=\nabla_{\vec n_\Sigma(p)}\vec n_{\partial M}.
\end{equation*}
Note that the first and the third terms in the left-hand side of the above are tangent to $\Sigma_0$. Thus, multiplying both sides by the unit normal vector $\vec n_0(p)$ yields
\[\frac{\partial\zeta}{\partial t}(p,0)=g\big(\nabla_{\vec n_0(p)}\vec n_{\partial M},\vec n_0(p)\big)=-\sff^{\partial M}\big(\vec n_0(p),\vec n_0(p)\big).\qedhere\]
\end{proof}
\begin{lemma}\label{thm:rightinverse}
There exists a bounded linear map $C^{1,\alpha}(\partial\Sigma)\ni g\mapsto\mathcal F_g\in C^{2,\alpha}(\Sigma)$ such that
$\mathcal F_g\equiv0$ on $\partial\Sigma$ and $\vec n_{\partial M}(\mathcal F_g)=g$.
\end{lemma}
\begin{proof}
The result is obtained with a straightforward partition of unity argument from an extension theorem of Whitney type, see for instance \cite[Ch.\ VI]{stein}, which gives an analogous result in $\mathds R^n$.
For the reader's convenience, a direct and elementary proof of Lemma~\ref{thm:rightinverse} is given in Appendix~\ref{app:whitney}.
\end{proof}

\section{Deformation of free boundary CMC hypersurfaces}\label{sec:TFI-CMC}

In this section, we give the detailed statement and proof of our main deformation result for free boundary CMC hypersurfaces. This will be a corollary of a slightly  more general result, concerning deformations of a free boundary CMC hypersurface by families of free boundary hypersurfaces with {\em prescribed mean curvature function}. We consider deformations parameterized by some smooth manifold $\Lambda$, with varying background metrics $g_\lambda$, $\lambda\in\Lambda$, having a varying group of isometries. A crucial assumption for our theory is a sort of smooth dependence of the ambient symmetries relatively to the parameter $\lambda$, which is stated in terms of smooth dependence of the Killing fields of $g_\lambda$. The construction has strong analogies with the notion of smooth bundles of Lie groups, introduced in \cite{UmeYam} in order to study deformations of immersed CMC tori from $\mathds R^3$ to other space forms. The notion was further developed in \cite{TasUmeYam} to study isometry groups of deformation of symmetric spaces. More notation and terminology are needed to formulate our result precisely.

\subsection{Smooth families of Killing fields and equivariant nondegeneracy}
In what follows, we allow for varying Riemannian metrics on $M$, and to this aim we consider the following setup:
\begin{itemize}
\item[(a)] $\Lambda$ is a differentiable manifold;
\item[(b)] $(g_\lambda)_{\lambda\in\Lambda}$ is a smooth family of Riemannian metrics on $M$;
\item[(c)] each $g_\lambda$ defines the same normal bundle  $T(\partial M)^\perp$ of $\partial M$.
\end{itemize}
Let us now assume that a smooth family $\Lambda\ni\lambda\mapsto\big\{K_1^\lambda,\ldots,K_r^\lambda\big\}$ of (local) vector fields on $M$ is given, and let $\Sigma\subset M$ be a compact hypersurface which is contained in the common domain of the $K_i^\lambda$'s and transversely oriented (i.e., its normal bundle is oriented). As usual, we denote by $\vec n_\Sigma^\lambda$ the $g_\lambda$-unit positively oriented normal field along $\Sigma$,
and by $\psi_i^\lambda$ the smooth function on $\Sigma$ defined by $\psi_i^\lambda=g_\lambda\big(\vec n_\Sigma^\lambda,K_i^\lambda\big)$.

We are interested in the above setup mainly under the assumptions that:
\begin{itemize}
\item each $K_i^\lambda$ is a (local) Killing field for $(M,g_\lambda)$;
\item for a given $\lambda_0\in\Lambda$, $\Sigma$ is a free boundary CMC hypersurface of $(M,g_{\lambda_0})$;
\item the maps $\psi_i^{\lambda_0}$, $i=1,\ldots,r$, form a basis of the space of Killing--Jacobi fields along $\Sigma$.
\end{itemize}
In this situation, we say that $\{K_1^\lambda,\ldots,K_r^\lambda\}_{\lambda\in\Lambda}$ \emph{is a smooth frame of Killing--Jacobi
fields generators along $\Sigma$} near $\lambda_0$.

\subsection{Mean curvature deformations}
With this setup, we are now ready to formulate our main result.

\begin{theorem}\label{thm:main}
Let $(g_\lambda)_{\lambda\in\Lambda}$ be a smooth family of Riemannian metrics as above, and
let $\Sigma_0 \subset M$ be an admissible hypersurface of $M$ that bounds a finite volume.
Let $\lambda_0\in\Lambda$ be fixed, and assume that $\Sigma_0$ is an equivariantly nondegenerate free boundary CMC embedding in $(M,g_{\lambda_0})$,
with mean curvature $h_0\in\mathds R$.
Assume that $\Lambda\ni\lambda\mapsto\{K_1^\lambda,\ldots,K_r^\lambda\}$ is a smooth frame of Killing--Jacobi
fields generators along $\Sigma$ near $\lambda_0$.

Then, there exists a neighborhood  ${\mathcal U}$ of  $(h_0,\lambda_0)$ in $\mathds R\times\Lambda$
and  a smooth family ${\mathcal U} \ni (h,\lambda)\mapsto\Sigma_{(h,\lambda)}\subset M$ of admissible hypersurfaces,  such that:
\begin{itemize}
\item for all $(h,\lambda)\in\mathcal U$, $\Sigma_{(h,\lambda)}$ is a free boundary hypersurface in $(M,g_\lambda)$ diffeomorphic to $\Sigma_0$ and with constant mean curvature equal to $h$;
\item $\Sigma_{(h_0,\lambda_0)}=\Sigma_0$.
\end{itemize}
Moreover, the family $\Sigma_{(h,\lambda)}$ is unique modulo congruence, i.e., if $\lambda'$ is sufficiently close to $\lambda_0$ and $\Sigma'$ is a free boundary  hypersurface of $(M,g_{\lambda'})$ sufficiently close to $\Sigma_{0}$, with  mean curvature function $h'\in\mathds R$ sufficiently $C^0$-close to the (constant) function $h_0$, then $\Sigma'$ is congruent to $\Sigma_{(h',\lambda')}$.
\end{theorem}

\begin{proof}
The strategy is to use an Implicit Function Theorem on the space of unparametrized normal embeddings\footnote{Since $g_\lambda$ define the same $T(\partial M)^\perp$, the notion of normal embedding is independent on $\lambda$.} $\mathcal E_\partial^\perp(\Sigma,M)$ discussed in Section~\ref{sec:normalhypersurfaces}, together with a gauge argument using the flux of Killing fields discussed in Section~\ref{sec:killingflux}.

Denote by $x_0\colon\Sigma\hookrightarrow M$ the smooth embedding with image $\Sigma_0$, and by $\vec n_0$ the inward (pointing to the region bounded by $\Sigma_0$) unit normal along $\Sigma_0\subset (M,g_{\lambda_0})$.
A sufficiently $C^{2,\alpha}$-small neighborhood of $[x_0]$ in $\mathcal E_\partial^\perp(\Sigma,M)$ is identified with a smooth submanifold
$\mathcal N \subset C^{2,\alpha}(\Sigma)$, as in Proposition~\ref{prop:normhyperman}. For $f\in\mathcal N$, we denote by $x_f\colon\Sigma\hookrightarrow M$ the corresponding embedding and by $\Sigma_f$ its image.
Furthermore, let $\vec n^f_\lambda$ be the inward unit normal along $\Sigma_f\subset(M,g_\lambda)$, and $H^f_\lambda$ be the mean curvature function of $\Sigma_f\subset(M,g_\lambda)$.

Consider the smooth function
\begin{equation}\label{eq:deftildeH}
\begin{aligned}
\mathcal H\colon\mathcal N\times\R^r\times\Lambda&\longrightarrow C^{0,\alpha}(\Sigma)\times\Lambda\\
\mathcal H(f,\alpha,\lambda)&:=\left(H^f_\lambda+\sum_{i=1}^r\alpha_i\,g_\lambda(K^i_\lambda,\vec n^f_\lambda),\lambda\right),
\end{aligned}
\end{equation}
where $K^i_\lambda$, $i=1,\dots,r$, are the local Killing vector fields in $(M,g_\lambda)$ that extend the Killing-Jacobi fields along $\Sigma_0$.

\begin{claim}\label{claim:Hsubm}
 $\mathcal H$ is a submersion near $(0,0,\lambda_0)\in\mathcal N\times\R^r\times\Lambda$.
\end{claim}

Differentiating \eqref{eq:deftildeH} at $(0,0,\lambda_0)$ in the direction $(\psi,\beta,\mu)$, with
$\psi\in C^{2,\alpha}_\partial(\Sigma)$, $\beta\in\R^r$ and $\mu\in T_{\lambda_0}\Lambda$, gives
\begin{equation}\label{eq:derHtilde}
\mathrm d\mathcal H(0,0,\lambda_0)[\psi,\beta,\mu]=\left(J_{\lambda_0}\psi+\sum_{i=1}^r\beta_i\,g_{\lambda_0}\big(K^i_{\lambda_0},\vec n^0_{\lambda_0}\big)+L_{\lambda_0}\mu\,,\,\mu\right)\!,
\end{equation}
where $J_{\lambda_0}$ is the Jacobi operator of the CMC hypersurface $\Sigma_0\subset (M,g_{\lambda_0})$, and
\begin{equation}\label{eq:linearizelambda}
L_{\lambda_0}\colon T_{\lambda_0}\Lambda\longrightarrow C^{0,\alpha}(\Sigma)
\end{equation}
is a linear operator given by the derivative of the map $\Lambda\ni\lambda\mapsto H^0_\lambda\in C^{0,\alpha}(\Sigma)$ at $\lambda_0$. Note  that the linear map
\[C^{2,\alpha}_\partial(\Sigma)\times\R^r\ni (\psi,\beta)\longmapsto J_{\lambda_0}\psi+\sum_{i=1}^r\beta_i\,g_{\lambda_0}\big(K^i_{\lambda_0},\vec n^0_{\lambda_0}\big)\in
C^{2,\alpha}(\Sigma)\]
is surjective. This follows  from the fact that $J_{\lambda_0}\colon C^{2,\alpha}_\partial(\Sigma)\to C^{0,\alpha}(\Sigma)$ is
an $L^2$-symmetric Fredholm operator of index $0$,
and the functions $g_{\lambda_0}\big(K^i_{\lambda_0},\vec n^0_{\lambda_0}\big)$ span its kernel. From \eqref{eq:derHtilde}, this implies that $\mathrm d\mathcal H(0,0,\lambda_0)$ is surjective.
It also follows readily from \eqref{eq:derHtilde} that:
\[\ker\dd\mathcal H(0,0,\lambda_0)=\ker J_{\lambda_0}\times\{0\}\times\{0\},\]
which is a subspace of dimension $r$
in $C^{2,\alpha}_\partial(\Sigma)\times\R^r\times T_{\lambda_0}\Lambda$.
This completes the proof of Claim~\ref{claim:Hsubm}.

Since submersions admit local sections, there exists a neighborhood $\widetilde{\mathcal U}$ of $(h_0,\lambda_0)$ in
$C^{0,\alpha}(\Sigma) \times \Lambda$
 and a smooth map $\widetilde{\mathcal U} \ni (h,\lambda)\mapsto\big(f(h,\lambda),\alpha(h,\lambda),\lambda\big)\in\mathcal N\times\R^r\times\Lambda$ such that
\begin{itemize}
\item[(a)] $f(h_0,\lambda_0)=0$
\item[(b)] $\mathcal H\big(f(h,\lambda),\alpha(h,\lambda),\lambda\big)=(h,\lambda)$, for all $(h,\lambda)\in\mathcal U$.
\end{itemize}
Now, define 
 $\mathcal U=\widetilde{\mathcal U}\cap(\mathds R\times\Lambda)=\big\{(h,\lambda)\in\widetilde{\mathcal U}:\text{$h$ is constant}\big\}$.
Since $\Sigma_0$ is an admissible hypersurface that bounds a finite volume, if $\widetilde{\mathcal U}$ is sufficiently small
we can assume that $\Sigma_{f(h,\lambda)}$ is also an admissible hypersurface that bounds a finite volume for all $(h,\lambda)\in\widetilde{\mathcal U}$, see Remark~\ref{rem:stability}.
By Corollary~\ref{thm:Htildeconstant} applied to $\Sigma_{f(h,\lambda)}$,
when $(h,\lambda)\in\mathcal U$, then $\mathcal H\big(f(h,\lambda),\alpha(h,\lambda),\lambda\big)=(h,\lambda)$ implies $\alpha(h,\lambda)=0$, and so $H^f_\lambda  = h$.
The desired smooth family is then  given by $\Sigma_{(h,\lambda)}:=x_{f(h,\lambda)}(\Sigma)$.

As to the uniqueness modulo congruence,
 we first recall that for $(h, \lambda) \in \mathcal U$,
 $\mathcal H^{-1}(h,\lambda)$ is identified with the set of $g_\lambda$-free boundary  hypersurfaces of $M$ near $\Sigma_0$ that are diffeomorphic to $\Sigma_0$ and with constant mean curvature equal to $h$, while the $f(h,\lambda)$-orbit is identified with the set of hypersurfaces of $M$ near $\Sigma_0$ that are congruent to $\Sigma_{(h,\lambda)}$. We shall obtain a proof of uniqueness modulo congruence by showing that these two sets coincide near $\Sigma_{(h,\lambda)}$.

Observe that for $(h,\lambda)$ near $(h_0,\lambda_0)$, the inverse image $\mathcal H^{-1}(h,\lambda)$ is a smooth submanifold of $\mathcal N\times\{0\}\times\{\lambda\}$ with dimension $r=\dim\ker\dd\mathcal H(0,0,\lambda_0)$.
On the other hand, for $(h,\lambda)$ sufficiently close to $(h_0,\lambda_0)$,
the functions $\psi^i_\lambda=g_\lambda\big(K^i_\lambda,\vec n^{f(\lambda,h)}_\lambda\big)$ are linearly independent by continuity. This implies that the $f(h,\lambda)$-orbit of the local action on $\mathcal N$ by the pseudo-group of local $g_\lambda$-isometries generated by the $K^i_\lambda$ is a smooth manifold of dimension $\geq r$, see Proposition~\ref{prop:normhyperman}. Clearly, if $f'$ belongs to this orbit, then $(f',0,\lambda)\in\mathcal H^{-1}(h,\lambda)$.
This implies that the $f(h,\lambda)$-orbit has in fact dimension equal to $r$, and that a sufficiently small neighborhood of $\big(f(h,\lambda),0,\lambda\big)$ in $\mathcal H^{-1}(h,\lambda)$ contains only elements of the
$f(h,\lambda)$-orbit. In other words, any free boundary CMC hypersurface $\Sigma'$, with mean curvature equal to $h$, and sufficiently close to $\Sigma_0$, must be congruent to $\Sigma_{(h,\lambda)}$.
\end{proof}

\section{Foliations by CMC hypersurfaces}\label{sec:foliation}
A direct consequence of Theorem~\ref{thm:main} is that an equivariantly nondegenerate free boundary
CMC hypersurface $\Sigma_0$ with mean curvature $h_0$ in $(M,g_{\lambda_0})$ belongs to a $1$-parameter family of free boundary CMC hypersurfaces $\{\Sigma_r\}_{r\in(-\varepsilon,\varepsilon)}$, with mean curvature $h_0+r$. A very natural geometric question is to determine whether a tubular neighborhood of $\Sigma_0$ in $(M,g_{\lambda_0})$ has a \emph{foliation} by free boundary CMC hypersurfaces. In particular, since $\Sigma_r$ is unique modulo congruence, this amounts to determining whether there exist ambient isometries $I_r$ of $(M,g_{\lambda_0})$, with $I_0=\id$, such that
$\{I_r(\Sigma_r)\}_{r\in(-\varepsilon,\varepsilon)}$ is a foliation.

Foliations by CMC hypersurfaces are important geometric objects, with deep ramifications in mathematical physics, see \cite{hy,ye-gr}. Deep contributions regarding the problem of foliating a tubular neighborhood of a submanifold with CMC hypersurfaces have been given by Ye~\cite{ye} and Mazzeo and Pacard~\cite{mp-foliations}. In the former, it is proved that the foliation by geodesic spheres of small radius with center at a nondegenerate critical point of the scalar curvature function can be perturbed to a foliation by CMC spheres. The latter considered simple closed nondegenerate geodesics, proving that a tubular neighborhood can be \emph{partially foliated} with CMC hypersurfaces obtained by perturbing tubes around the geodesic. The main difficulty in obtaining an actual foliation in this case, as well as in higher dimensions, is related to a bifurcation phenomena as the CMC hypersurfaces collapse, see \cite{bp,mmp}.

\subsection{Nondegenerate \texorpdfstring{$\pmb\Sigma_{\pmb0}$}{hypersurface}}
Initially, let us make the simplifying assumption that $\Sigma_0$ is actually nondegenerate, i.e., there are no nonzero Jacobi fields $\psi$ along $\Sigma_0$ satisfying \eqref{eq:linearizedfreebdy}. In this situation,
the Jacobi operator $J_0\colon C^{2,\alpha}_\partial(\Sigma_0)\to C^{0,\alpha}(\Sigma_0)$ is an isomorphism, hence there
exists a unique $\psi\in C^{2,\alpha}_\partial(\Sigma_0)$ satisfying
\begin{equation}\label{eq:Lpsi=1}
J_0(\psi)\equiv1.
\end{equation}
By elliptic regularity, such a function $\psi$ is smooth.

\begin{proposition}\label{prop:Lpsi=1}
If the solution $\psi$ to \eqref{eq:Lpsi=1} does not have zeros in $\Sigma_0$, then the $1$-parameter
family $\{\Sigma_r\}_{r\in(-\varepsilon,\varepsilon)}$ above is a foliation of a neighborhood of $\Sigma_0$.
\end{proposition}

\begin{proof}
Let $\phi_r\in C^{2,\alpha}_\partial(\Sigma_0)$ be such that $\Sigma_r=\exp(\phi_r\, \vec n_{\Sigma_0})$; in particular, $\phi_0=0$.
Consider the induced variational field along $\Sigma_0$, given by
$V=\big(\frac{\dd}{\dd r}\big\vert_{r=0}\phi_r\big)\,\vec n_{\Sigma_0}$. Moreover,
let $\psi_V:=g(V,\vec n_{\Sigma_0})=\frac{\dd}{\dd r}\big\vert_{r=0}\phi_r$. Then,
\[J_0(\psi_V)=\tfrac{\mathrm d}{\mathrm dr}\big\vert_{r=0}(h_0+r)=1.\]
By uniqueness, $\psi_V=\psi$. Since $\Sigma_0$ is embedded,
$\Sigma_0\times\R\ni(p,t)\mapsto\exp_p\big(t\,\vec n_{\Sigma_0}(p)\big)\in M$ gives a diffeomorphism
from a neighborhood of $\Sigma_0\times\{0\}$ in $\Sigma_0\times\R$ onto a neighborhood of $\Sigma_0$ in $M$.
By composition, the map $(p,r)\mapsto\exp_p\big(\phi(p)\,\vec n_{\Sigma_0}(p)\big)$ is a diffeomorphism from
$\Sigma_0\times(-\varepsilon,\varepsilon)$ onto a neighborhood $U$ of $\Sigma_0$ in $M$, for $\varepsilon>0$
sufficiently small.\footnote{The assumption that $\psi=\frac{\dd}{\dd r}\big\vert_{r=0}\phi_r$ does not vanish
on $\Sigma_0$ and compactness of $\Sigma_0$ imply that the map
$\Sigma_0\times\mathds R\ni(p,r)\mapsto\big(p,\phi_r(p)\big)\in\Sigma_0\times\mathds R$
is a diffeomorphism between two neighborhoods of $\Sigma_0\times\{0\}$ in $\Sigma_0\times\R$.}
Under this diffeomorphism, the hypersurfaces $\Sigma_r$ correspond to the slices $\Sigma_0\times\{r\}$,
which form a foliation of $U$.
\end{proof}

\subsection{Equivariantly nondegenerate \texorpdfstring{$\pmb\Sigma_{\pmb0}$}{hypersurface}}
Let us now consider the more general case in which $\Sigma_0$ is equivariantly  nondegenerate (recall Definition~\ref{def:equivnondeg}). By Proposition~\ref{thm:integralidentities}, the constant function $1$ is $L^2$-orthogonal to every Killing-Jacobi field, and therefore belongs to the image of the Jacobi operator $J_0$. In particular, equation \eqref{eq:Lpsi=1} has solutions in $C^{2,\alpha}_\partial(\Sigma_0)$.

\begin{proposition}\label{prop:foliatedeqnondeg}
If  there exists a nonvanishing solution of  \eqref{eq:Lpsi=1},
then there exists a neighborhood of $\Sigma_0$ foliated by a smooth $1$-parameter family of
free boundary CMC hypersurfaces $\{I_r(\Sigma_r)\}_{r\in (-\varepsilon,\varepsilon)}$, as described above.
\end{proposition}

\begin{proof}
It suffices to show that, given a nonvanishing solution $\psi$ to \eqref{eq:Lpsi=1},
there exists a smooth $1$-parameter family of maps $\varphi_r\in C^{2,\alpha}_\partial(\Sigma_0)$,
with $\varphi_0=0$ and $\frac{\dd}{\dd r}\big\vert_{r=0}\varphi_r=\psi$, such that the
hypersurface $\exp(\varphi_r\,\vec n_{\Sigma_0})$ has constant mean curvature $h_0+r$, for
$\vert r\vert$ sufficiently small. Once the existence of such family $\varphi_r$ is established,
the proof follows exactly as in the nondegenerate case (Proposition~\ref{prop:Lpsi=1}).

Applying Theorem~\ref{thm:main} with a fixed metric $g_{\lambda_0}$, we get the existence of a smooth $1$-parameter
family of maps $\phi_r\in C^{2,\alpha}_\partial(\Sigma_0)$, with $\phi_0=0$, such that
$\Sigma_r=\exp(\phi_r\,\vec n_{\Sigma_0})$ has constant mean curvature $h_0+r$, for $\vert r\vert$ sufficiently
small. Note that the function $\widetilde\psi=\frac{\dd}{\dd r}\big\vert_{r=0}\phi_r$ is another solution
to \eqref{eq:Lpsi=1}. Thus, there exists a local Killing vector field $K$ such that $\psi=\widetilde\psi+\psi_K$,
where $\psi_K$ is the Killing-Jacobi field \eqref{eq:defKillJac}. Denote by
$\{I_r\}_{r\in(-\varepsilon,\varepsilon)}$ the local flow of $K$.
Since $K$ is Killing, $I_r(\Sigma_r)$ is a smooth perturbation of $\Sigma_0$ with constant mean curvature $h_0+r$.
The corresponding variational field is
\small
\begin{equation}\label{eq:varfieldV+K}
\tfrac{\dd}{\dd r}\big\vert_{r=0}\, I_r\left(\exp(\phi_r\,\vec n_{\Sigma_0})\right)
 = \tfrac{\dd}{\dd r}\big\vert_{r=0}\Big[I_r\left(\exp(\phi_0\,\vec n_{\Sigma_0})\right)+
I_0\left(\exp(\phi_r\,\vec n_{\Sigma_0})\right)\Big] = K+V,
\end{equation}
\normalsize
where
$V=\tfrac{\dd}{\dd r}\big\vert_{r=0}\left(\exp(\phi_r\,\vec n_{\Sigma_0})\right)=
\widetilde\psi\,\vec n_{\Sigma_0}$ is the variational vector field corresponding to $\Sigma_r$.
By Proposition~\ref{prop:normhyperman}, there exists  a smooth $1$-parameter family of maps $\varphi_r\in
C^{2,\alpha}_\partial(\Sigma_0)$, with $\varphi_0=0$, such that $I_r(\Sigma_r)=\exp(\varphi_r\,\vec n_{\Sigma_0})$.
Finally, from \eqref{eq:varfieldV+K}:
\[\tfrac{\dd}{\dd r}\big\vert_{r=0}\varphi_r=g\big(K+V,\vec n_{\Sigma_0}\big)=\psi_K+\widetilde\psi=\psi.\qedhere\]
\end{proof}

\begin{remark}\label{rem:perturbfoliation}
The condition $\psi\ne0$ for solutions $\psi$ of \eqref{eq:Lpsi=1} is open in the $C^0$-topology, since $\Sigma_0$ is compact. Therefore,  if $r\mapsto\Sigma_r=\exp(\phi_r\cdot\vec n_{\Sigma_0})$ is a perturbation  of $\Sigma_0$ by free boundary CMC hypersurfaces, with $H_{\Sigma_r}=h_0+r$ for all $r$, and with $\psi:=\frac{\mathrm d}{\mathrm dr}\big\vert_{r=0}\phi_r\ne0$ on $\Sigma_0$,
then any other  perturbation $\widetilde\Sigma_r$ of $\Sigma_0$  sufficiently $C^1$-close to $\Sigma_r$ must also foliate a neighborhood of $\Sigma_0$.
\end{remark}

\section{Free boundary disks in the unit ball}\label{sec:disks}
The simplest free boundary minimal hypersurface in the unit ball $B^{n+1}\subset\R^{n+1}$ is the flat disk $D^n$ obtained by intersecting $B^{n+1}$ with a hyperplane in $\R^{n+1}$, which we may assume to be $\R^n=\{e_{n+1}\}^\perp$. In this section, as a first example, we apply Theorem~\ref{thm:main} to verify that $D^n$ can be deformed to free boundary CMC hypersurfaces inside the unit ball of other $(n+1)$-dimensional space forms. We remark that stability issues for such surfaces have been studied in \cite{ros08,ros-souam,RosVergasta,souam}.

\subsection{Equivariant nondegeneracy}
We now verify the main hypothesis needed to apply Theorem~\ref{thm:main}, regarding equivariant nondegeneracy of the flat disk. See Lemma~\ref{lemma:catenoidnondegenerate} below for a generalization to other surfaces of revolution.

We begin by noting that the space forms are conformal, and thus define the same normal bundle $T(\partial M)^\perp$.

\begin{lemma}\label{lem:disknondeg}
The flat disk $D^n$ is an equivariantly nondegenerate free boundary minimal hypersurface in $B^{n+1}$, with nullity
equal to $n$.
\end{lemma}

\begin{proof}
The Jacobi operator \eqref{eq:Jacobiop} of $D^n$ is simply the (nonnegative) Laplacian of the flat metric, as $D^n\subset B^{n+1}$ is totally geodesic and the ambient curvature vanishes. Moreover, the linearized free boundary condition \eqref{eq:linearizedfreebdy} reads $\langle\nabla\psi(x),x\rangle=\psi(x)$, since $\partial B^{n+1}=S^n$ is the unit sphere in $\R^{n+1}$. Altogether, Jacobi fields along $D^n$ are harmonic functions $\psi\colon D^n\to\R$ with Robin boundary conditions:
\begin{equation}\label{eq:robin}
\begin{cases}
\Delta\psi=0 & \text{in } D^n,\\
\frac{\partial\psi}{\partial \vec n}=\psi & \text{on } \partial D^n.
\end{cases}
\end{equation}
We claim that the space of solutions to \eqref{eq:robin} is spanned by the coordinate functions:
\begin{equation}\label{eq:heightfunctions}
f_i(x):=\langle x,e_i\rangle, \quad i=1,\dots,n,
\end{equation}
where $\{e_i\}$ is an orthonormal basis of $\R^n$. Indeed, expanding a solution $\psi=\sum_k \psi_k$ as a sum of homogeneous polynomials $\psi_k$ of degree $k$ (e.g., using the Taylor series at $x=0$), we have that $\frac{\partial\psi}{\partial \vec n}=\sum_k \langle \nabla\psi_k(x),x\rangle=\sum_k k \, \psi_k$, which can only be equal to $\psi$ provided $\psi_k=0$ for all $k\ne1$. In other words, $\psi$ must be a homogeneous polynomial of degree $1$, hence a linear combination of \eqref{eq:heightfunctions}. In particular, the nullity of $D^n$ is equal to $n$.

The space of Killing fields of the flat ball $B^{n+1}$ is identified with the Lie algebra $\mathfrak o(n+1)$ of the orthogonal group $\mathsf O(n+1)$, as each $A\in\mathfrak o(n+1)$ corresponds to the Killing field $K_A(x)=A\,x$, for $x\in B^{n+1}\subset\R^{n+1}$. Killing-Jacobi fields on $D^n$ are given by the restriction to $D^n$ of the linear map
\begin{equation}\label{eq:killing-jacobi-disk}
x\mapsto\langle K_A(x), e_{n+1}\rangle=\sum_{j=1}^n a_{n+1,j}\,x_j
\end{equation}
where $A=(a_{ij})$. Since the coefficients $a_{n+1,j}$, $1\leq j\leq n$, are arbitrary real numbers, there are $n$ linearly independent Killing-Jacobi fields\footnote{Geometrically, these are infinitesimal rotations of $B^{n+1}$ with rotation axis tangent to $D^n$.} along $D^n$, which is hence equivariantly nondegenerate.
\end{proof}

\begin{remark}\label{rem:noradialJacobifields}
From \eqref{eq:killing-jacobi-disk}, we observe that there are no nontrivial  rotationally invariant Jacobi fields along the $n$-disk $D^n$.
\end{remark}

\subsection{Killing fields in space forms}
In order to apply Theorem~\ref{thm:main} to deform $D^n$, we still have to verify that the above Killing fields of the flat ambient metric extend to a smooth family of Killing fields of ambient metrics with constant curvature.

Let $M^{n+1}_\lambda$ be the simply-connected space form of constant curvature $\lambda$. We denote by $B^{n+1}_\lambda$ the unit ball in $M^{n+1}_\lambda$, which is a Riemannian manifold with boundary,\footnote{For simplicity, we only consider $\lambda<\pi^2$, so that the unit ball in $M^{n+1}_\lambda$ remains a manifold with boundary. Of course, higher values of $\lambda>0$ can be achieved by rescaling; considering instead of a unit ball, a ball of fixed radius $<\frac{\pi}{\sqrt \lambda}$.}
equipped with the induced metric $g_\lambda:=\dd r^2+\snsq(r)\,\dd s^2_n$, where $\dd s^2_n$ is the standard metric on the unit sphere $S^n$ and
\begin{equation*}
\sn(r)=\begin{cases}
\frac{1}{\sqrt\lambda}\sin\big(r\sqrt\lambda\big) &\text{if } \lambda>0,\\
r & \text{if } \lambda=0,\\
\frac{1}{\sqrt{-\lambda}}\sinh\big(r\sqrt{-\lambda}\big) &\text{if } \lambda<0.
\end{cases}
\end{equation*}

\begin{lemma}\label{lemma:killingconstcurv}
Every Killing field $K_0$ of $B^{n+1}_0$ belongs to a smooth $1$-parameter family $K_\lambda$ of Killing fields of $B^{n+1}_\lambda$.
\end{lemma}

\begin{proof}
It is a well-known fact that the isometry groups $\mathsf G_\lambda$ of $M^{n+1}_\lambda$ form a smooth bundle of Lie groups over $\R$, and the corresponding Lie algebras $\mathfrak g_\lambda$ form a smooth bundle of Lie subalgebras of $\mathfrak{gl}(n+1)$ over $\R$, see e.g.\ \cite{UmeYam} for $n=2$ and \cite[Ex.\ 2.6]{BetPicSic2014} for $n\geq2$. The isometry group of the corresponding unit ball $B^{n+1}_\lambda\subset M^{n+1}_\lambda$ is the subgroup $\mathsf H_\lambda$ of $\mathsf G_\lambda$ formed by isometries that fix the origin $0\in B^{n+1}_\lambda$, and the corresponding subalgebras are $\mathfrak h_\lambda=\{K_\lambda\in\mathfrak g_\lambda:K_\lambda(0)=0\}$. As $\mathfrak h_\lambda$ is a smooth bundle of Lie subalgebras of $\mathfrak{gl}(n+1)$ over $\R$, it follows that every Killing field $K_0\in\mathfrak h_0$ of $B^{n+1}_0$ admits a smooth extension $K_\lambda\in\mathfrak h_\lambda$ to a Killing field of $B^{n+1}_\lambda$. Alternatively, the smooth dependence on $\lambda$ of solutions $K$ to $\mathcal L_K (g_\lambda)=0$ can be used to prove that the Killing field $K_0$ admits an extension $K_\lambda$.
\end{proof}

\subsection{Free boundary CMC disks}
We are now ready to apply Theorem~\ref{thm:main} to deform the flat disk $D^n\subset B^{n+1}$ in the flat unit ball through other CMC disks varying both the values of its mean curvature and of the ambient curvature.

\begin{proposition}\label{prop:deformationdisk}
The flat disk $D^n$ in $B^{n+1}_0$ is the member $\Sigma_{(0,0)}$ of a smooth $2$-parameter family of surfaces $\{\Sigma_{(h,\lambda)}\}$, where each $\Sigma_{(h,\lambda)}$ is a free boundary disk with constant mean curvature $h$ inside the unit ball $B^{n+1}_\lambda$ of constant curvature~$\lambda$.  Each $\Sigma_{(h,\lambda)}$ is invariant under the action of a subgroup of isometries of $B^{n+1}_\lambda$ isomorphic to $\mathsf O(n)$.
The family $\{\Sigma_{(h,\lambda)}\}$ is unique modulo congruence, and it is in fact uniquely determined by requiring that each $\Sigma_{(h,\lambda)}$ is invariant by rotations around the $x_{n+1}$-axis.
\end{proposition}

\begin{proof}
From Lemma~\ref{lem:disknondeg}, the flat disk $D^n\subset B^{n+1}_0$ is equivariantly nondegenerate. By Lemma~\ref{lemma:killingconstcurv}, every Killing field of $B^{n+1}_0$ belongs to a smooth $1$-parameter family of Killing fields of $B^{n+1}_\lambda$. Thus, the existence of $\Sigma_{(h,\lambda)}$ and its uniqueness modulo congruence follow from Theorem~\ref{thm:main}. In order to prove that $\Sigma_{(h,\lambda)}$ have the claimed symmetries, one applies the Implicit Function Theorem in the space of (unparametrized) embeddings into
$B^{n+1}_\lambda$ that are invariant under the action of the subgroup of isometries of $B^{n+1}_\lambda$ that fix the line spanned by $e_{n+1}$.
This subgroup is isomorphic to the subgroup $\mathsf O(n)$ of linear isometries of the hyperplane
$\R^n=\{e_{n+1}\}^\perp$. The flat disk is nondegenerate in this space,
as there are no non-trivial radial Jacobi fields along $D^n$, see Remark~\ref{rem:noradialJacobifields}.
This yields the existence of a smooth family $\Sigma'_{(h,\lambda)}$ with these symmetries, which must be congruent to $\Sigma_{(h,\lambda)}$ by uniqueness.
\end{proof}

\begin{remark}\label{rem:gaugefixing}
As in the proof of Proposition~\ref{prop:deformationdisk} (and also Proposition~\ref{prop:deformationcritcatenoid}), the standard Implicit Function Theorem may be used on the space of embeddings that are invariant under certain symmetries to obtain the existence of a deformation result for equivariantly nondegenerate free boundary CMC hypersurfaces. Nevertheless, this gauge-fixing approach does not guarantee the \emph{non-existence} of other CMC deformations \emph{without such symmetries}, which is only achieved with Theorem~\ref{thm:main}.
\end{remark}

\begin{remark}
It was recently proved by Fraser and Schoen~\cite{fraser-schoen4}, without any symmetry assumptions, that a free boundary minimal $2$-disk in the constant curvature ball $B_\lambda^{n+1}$, $n\geq3$, must be totally geodesic.  
\end{remark}

\begin{proposition}\label{prop:foliationdisk}
Let $\Sigma_s=\Sigma_{(h(s),\lambda(s))}$, $s\in[-\delta,\delta]$, be a smooth $1$-parameter subfamily of the rotationally invariant family
 $\{\Sigma_{(h,\lambda)}\}$ given in Proposition~\ref{prop:deformationdisk}, with $h(0)=0$, $h'(0)\neq0$ and $\lambda(0)=0$.
Then, for $\varepsilon>0$ sufficiently small, $\{\Sigma_s\}_{s\in (-\varepsilon,\varepsilon)}$ determines a foliation of a neighborhood of the flat disk $\Sigma_0=\Sigma_{(0,0)}$ in $B^{n+1}$.
\end{proposition}

\begin{proof}
We apply a slight generalization of the criterion in Section~\ref{sec:foliation} to verify that $\{\Sigma_s\}$ determines a foliation. Let $[-\delta,\delta]\ni s\mapsto\phi_s\in C^{2,\alpha}(\Sigma)$ be the smooth map such that $\Sigma_s=\exp(\phi_s\,\vec n_0)$, and set $\psi:=\frac{\dd}{\dd s}\big|_{s=0}\phi_s\in C^{2,\alpha}_\partial(D^n)$. Arguing as in Section~\ref{sec:foliation}, the conclusion will follow if we show that $\psi$ does not change sign on $D^n$.

Linearizing the equation $H_{\Sigma_s}=h(s)$  at $s = 0$, we find that $\psi$ is a solution to
\begin{equation}
h'(0)\, J_0\psi + \lambda'(0)\, L_0\psi = h'(0),
\end{equation}
where $L_0$ is the linear operator in \eqref{eq:linearizelambda}. Note that $\Sigma_{(0,\lambda)}$ is the fixed point set of the reflection about the hyperplane $\R^n=\{e_{n+1}\}^\perp$, and this involution is an isometry of $B^{n+1}_\lambda$ for all $\lambda$. Thus, $\Sigma_{(0,\lambda)}$ is totally geodesic (in particular, minimal) for all $\lambda$ and hence $L_0\psi=0$. Since $h'(0)\neq0$, we conclude that $\psi\in C^{k,\alpha}_\partial(D^n)$ is a solution to \eqref{eq:Lpsi=1}.
Similarly to \eqref{eq:robin}, it is easy to see that this is equivalent to
\begin{equation}\label{eq:Lpsi=1-disk}
\begin{cases}
\Delta\psi=1 & \text{in } B^{n+1},\\
\frac{\partial\psi}{\partial\vec n}=\psi & \text{on } \partial B^{n+1},
\end{cases}
\end{equation}
Since each $\Sigma_{(h(s),\lambda(s))}$ is invariant by rotations around the $x_{n+1}$-axis, then $\psi$ must be a radial function.

We claim that \eqref{eq:Lpsi=1-disk} admits a unique radial solution.
This follows easily from the fact that $D^n$ is nondegenerate as a rotationally invariant CMC free boundary hypersurface in $B^{n+1}$, see Remark~\ref{rem:noradialJacobifields}. Namely, the difference of any two radial solutions of \eqref{eq:Lpsi=1-disk} would be a radial Jacobi field along $D^n$. By inspection, one sees that $\tfrac{1}{2n}(|x|^2+1)$ is a radial solution of \eqref{eq:Lpsi=1-disk}, and therefore, by uniqueness, $\psi=\tfrac{1}{2n}(|x|^2+1)$. This function is positive on $D^n$, which concludes the proof.
\end{proof}

\begin{remark}
By Remark~\ref{rem:perturbfoliation}, the conclusion of Proposition~\ref{prop:foliationdisk} still holds when  the family $\{\Sigma_s\}$ does not contain surfaces that are invariant by rotations, but is \emph{$C^1$-close} to a family of hypersurfaces invariant under rotations around the $x_{n+1}$-axis.
\end{remark}
\begin{figure}[ht]
\centering
\vspace{-0.5cm}
\includegraphics[scale=0.5]{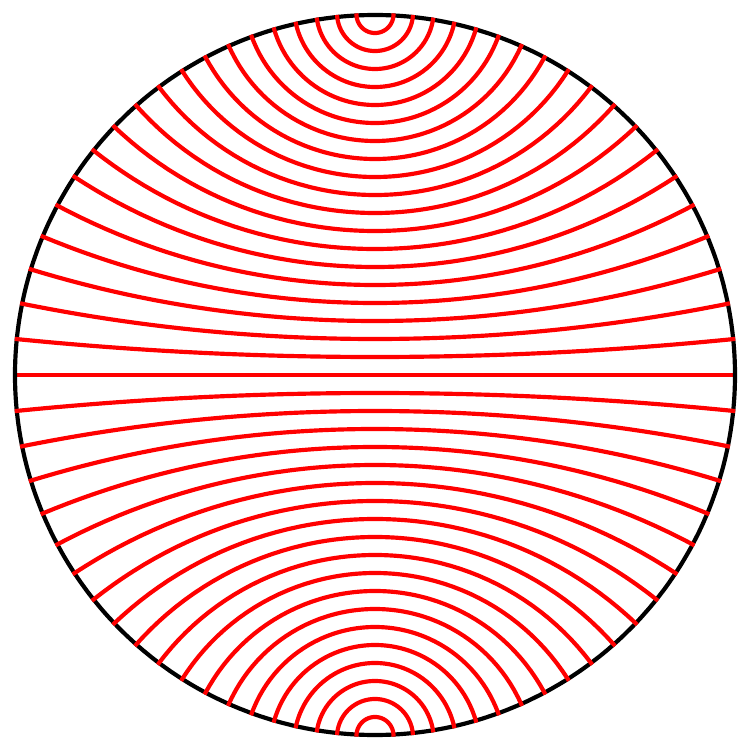}
\caption{Free boundary CMC spherical caps foliating the unit ball.}
\label{fig:sphcaps}
\end{figure}

\section{Free boundary Delaunay annuli in the unit ball}\label{sec:delaunay}

In this section, we apply our main result to prove that a well-studied free boundary minimal surface of the unit ball $B^3$ in Euclidean space, the critical catenoid, is a member of a $2$-parameter family of free boundary CMC surfaces inside the unit ball in a space form.
We then identify these surfaces in terms of Delaunay surfaces.

\subsection{Critical catenoid}
Recall that a \emph{catenoid} is a minimal surface in $\R^3$ obtained by rotating a catenary curve
\begin{equation}\label{eq:catenoid-generatrix}
x=c\cosh\frac{z}{c}, \quad c>0,
\end{equation}
about its directrix $x=0$. Besides affine planes, catenoids are the only minimal surfaces of revolution in $\R^3$. A simple argument shows that the portion of the catenoid generated by \eqref{eq:catenoid-generatrix} with $|z|\leq z_{c}$, where $z_c$ is the positive solution to
\begin{equation}
z_{c}=c\coth\frac{z_{c}}{c},
\end{equation}
is a normal minimal surface inside the ball of radius $r_c=\sqrt{z_{c}^2+c^2\cosh^2(z_{c}/c)}$ centered at the origin.
\begin{figure}[ht]
\centering
\includegraphics[scale=0.4]{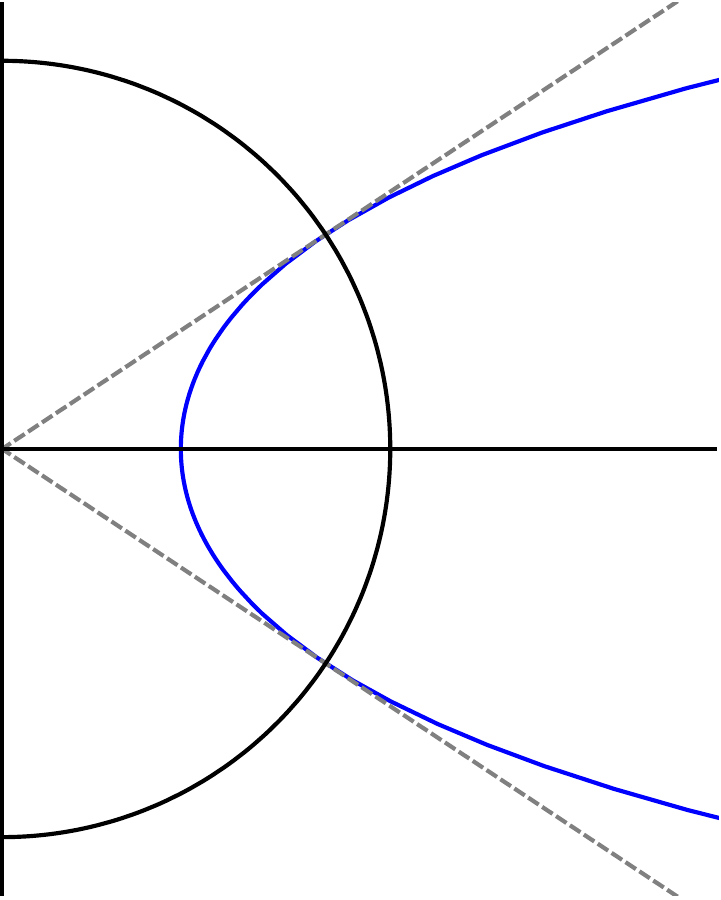}
\begin{pgfpicture}
\pgfputat{\pgfxy(0.2,3.05)}{\pgfbox[center,center]{$x$}}
\pgfputat{\pgfxy(-5,6.4)}{\pgfbox[center,center]{$z$}}
\pgfputat{\pgfxy(-2.1,3.25)}{\pgfbox[center,center]{$r_c$}}
\end{pgfpicture}
\caption{The catenary \eqref{eq:catenoid-generatrix} meets the circle of radius $r_c$ orthogonally.}
\label{fig:crit-catenoid}
\end{figure}
In particular, choosing $c>0$ such that $r_c=1$, the portion of the corresponding catenoid contained in the unit ball $B^3$ is a free boundary minimal surface, often referred to as the \emph{critical catenoid}.

\subsection{Equivariant nondegeneracy}
The key step to apply Theorem~\ref{thm:main} to deform the critical catenoid is to verify its equivariant nondegeneracy (see Definition~\ref{def:equivnondeg}).

\begin{lemma}\label{lemma:catenoidnondegenerate}
Let $\Sigma$ be a normal CMC surface of revolution  in $B^3$. If $\Sigma$ is symmetric with respect to a plane orthogonal to the rotation axis, then its nullity is either $2$ or $3$.
Furthermore, the critical catenoid admits no nontrivial rotationally symmetric Jacobi field satisfying \eqref{eq:linearizedfreebdy}, has nullity $2$ and is equivariantly nondegenerate.
\end{lemma}

\begin{proof}
We may assume that the rotation axis is the line $x=y=0$, and that the orthogonal plane is the coordinate plane $z=0$. The space of Killing-Jacobi fields along $\Sigma$ is two-dimensional, and spanned by the functions
\begin{equation}\label{eq:killingjacobi-catenoid}
f_i:=\langle e_i\times X,\vec n_\Sigma\rangle, \quad i=1,2,
\end{equation}
where $e_1$ and $e_2$ denote respectively the constant fields in the direction of the $x$-axis and of the $y$-axis, and $X$ the position vector. In other words, $f_1$ and $f_2$ are infinitesimal rotations around the $x$-axis and $y$-axis, respectively. Note that each of these Killing-Jacobi fields:
\begin{enumerate}
\item\label{item:jacobiA} has exactly two nodal domains (see Figure~\ref{fig:2nodaldomains});
\item\label{item:jacobiB} is not rotationally invariant.
\end{enumerate}
\begin{figure}[ht]
\centering
\includegraphics[width=0.35\textwidth]{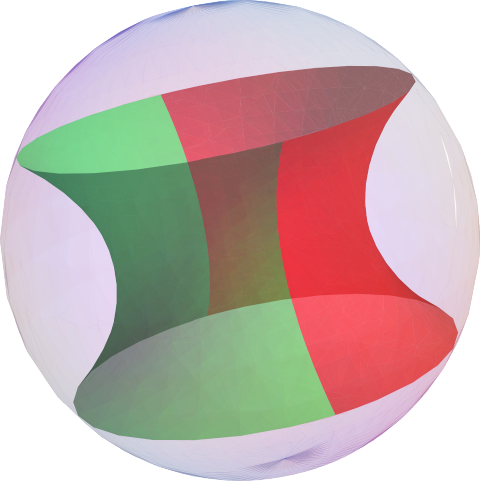}
\caption{Two nodal domains of the Jacobi fields $f_1$ and $f_2$.}
\label{fig:2nodaldomains}
\end{figure}
Let $\big[-\frac12L_\gamma,\frac12L_\gamma\big]\ni s\mapsto\gamma(s)=\big(x(s),z(s)\big)$ be the unit speed parametrization of the generatrix of $\Sigma$ in the $xz$-plane. By the symmetry assumptions, $x>0$ is an even function, and $z$ is odd. Parametric equations for $\Sigma$ are hence given by:
\begin{equation*}
\big(x(s)\cos\theta,\,x(s)\sin\theta,\,z(s)\big), \quad (s,\theta)\in\big[-\tfrac12 L_\gamma, \tfrac12 L_\gamma\big]\times\big[0,2\pi\big).
\end{equation*}
A unit normal $\vec n_\Sigma\colon\Sigma\to S^2$ is given by:
\begin{equation}\label{eq:Gaussmap}
\vec n_\Sigma(s,\theta)=\big(\dot z(s)\cos\theta,\dot z(s)\sin\theta,-\dot x(s)\big).
\end{equation}
The Laplacian $\Delta_\Sigma$ of the pull-back metric is easily seen to be the operator
\[\Delta_\Sigma=-\frac1x\,\frac\partial{\partial s}\left(x\,\frac\partial{\partial s}\right)-\frac1{x^2}\,\frac{\partial^2}{\partial\theta^2},\]
hence the Jacobi operator \eqref{eq:Jacobiop} of $\Sigma$ is given by:
\begin{equation}\label{eq:JacobiSigma}
J_\Sigma=-\frac1x\,\frac\partial{\partial s}\left(x\,\frac\partial{\partial s}\right)-\frac1{x^2}\,\frac{\partial^2}{\partial\theta^2}
-\big\Vert\sff^\Sigma\big\Vert^2.
\end{equation}
By rotational symmetry, $\big\Vert\sff^\Sigma\big\Vert^2$ depends only on $s$, and can be explicitly computed as the sum of the squares of the principal curvatures of $\Sigma$,
\[\Vert\sff^\Sigma\Vert^2=\Vert\ddot\gamma\Vert^2+\frac{\dot z^2}{x^2}=(\dot x\ddot z-\ddot x\dot z)^2+\frac{\dot z^2}{x^2}.\]
The linearized free boundary conditions \eqref{eq:linearizedfreebdy} for a Jacobi field $\psi$ read:
\[\frac{\partial\psi}{\partial s}\big(-\tfrac12L_\gamma,\theta\big)+\psi\big(-\tfrac12L_\gamma,\theta\big)=0,\quad
\frac{\partial\psi}{\partial s}\big(\tfrac12L_\gamma,\theta\big)-\psi\big(\tfrac12L_\gamma,\theta\big)=0.\]
Separation of variables $\psi=S(s)\,\Theta(\theta)$ for the above problem gives the following pair of linear ODEs on $S$ and $\Theta$:
\begin{eqnarray}
-(xS')'+\left(\frac\kappa x-x\big\Vert\sff^\Sigma\big\Vert^2\right)S=0&\label{eq:eqforS} \\
 \Theta''+\kappa\Theta=0& \label{eq:eqforTheta}
\end{eqnarray}
and boundary conditions
\begin{equation}\label{eq:BVforS}
 S'\big(-\tfrac12 L_\gamma\big)+S\big(-\tfrac12 L_\gamma\big)=0,\quad
 S'\big(\tfrac12 L_\gamma\big)-S\big(\tfrac12 L_\gamma\big)=0.
 \end{equation}

In \eqref{eq:eqforS} and \eqref{eq:eqforTheta}, $\kappa$ is an arbitrary real constant.
 Equation \eqref{eq:eqforTheta} admits periodic solutions only when $\kappa=n^2$, for some nonnegative integer $n$. Hence, a basis for $\ker J_\Sigma$ is given by functions of the form
 $S_n(s)\cdot\sin(n\theta)$ and $S_n(s)\cdot\cos(n\theta)$, where $n$ is some nonnegative integer for which
 there exists a nontrivial solution $S_n$ of the boundary value problem
 \begin{equation}\label{eq:eqforSn2}
-(xS')'+\left(\frac{n^2}x-x\Vert\sff^\Sigma\Vert^2\right)S=0,
\end{equation}
with boundary conditions \eqref{eq:BVforS}.

Using the properties \eqref{item:jacobiA} and \eqref{item:jacobiB} of the Killing-Jacobi fields $f_i$, it is easy to see that $f_1$ and $f_2$ correspond respectively to $S_1\sin\theta$ and $S_1\cos\theta$, where $S_1$ is the solution of \eqref{eq:eqforSn2} with $n=1$, satisfying \eqref{eq:BVforS}.  Since $f_1$ and $f_2$ have exactly two nodal domains, it follows that $S_1$ has no zero in $\left[-\frac12L_\gamma,\frac12L_\gamma\right]$.

Consider \eqref{eq:eqforSn2} as a Sturm-Liouville equation
\begin{equation}\label{eq:SturmLiouville}
-(xS')'-x\Vert\sff^\Sigma\Vert^2S=\lambda\cdot\tfrac1xS,
\end{equation}
with $\lambda=-n^2$, $n\in\mathds N$. When $n=1$, the solution of \eqref{eq:SturmLiouville}
satisfying \eqref{eq:BVforS} is $S_1$. Since it has no zero in $\left[-\frac12L_\gamma,\frac12L_\gamma\right]$, it follows from Sturm-Liouville theory that the first eigenvalue of \eqref{eq:SturmLiouville} must be $\lambda_1=-1$. In particular, there is no nontrivial solution of \eqref{eq:SturmLiouville} satisfying \eqref{eq:BVforS} when $\lambda=-n^2$, $n>1$.
Thus, $\ker J_\Sigma$ is spanned by $f_1$, $f_2$ and at most one more rotationally symmetric Jacobi field satisfying \eqref{eq:BVforS}, corresponding to a solution of \eqref{eq:SturmLiouville} when $\lambda=0$. This shows that the nullity of $\Sigma$ is either $2$ or $3$.

Let us show that there cannot be a solution of \eqref{eq:SturmLiouville} with $\lambda=0$ satisfying \eqref{eq:BVforS} when $\Sigma$ is the critical catenoid. We claim that any such solution $S_0$ must be either an even or an odd function of $s$. Note that $\widetilde S_0(s):=S_0(-s)$ also solves \eqref{eq:SturmLiouville} and \eqref{eq:BVforS}. As the space of solutions has dimension $\leq 1$, we have $\widetilde S_0=\alpha\, S_0$, for some $\alpha\in\R$. Iteration gives $\alpha^2=1$, hence $\alpha=\pm1$, proving the claim. Furthermore, disregarding boundary conditions, \eqref{eq:SturmLiouville} admits two linearly independent solutions; one even and one odd. The odd solution is the rotationally symmetric Jacobi field $\nu_3$ along $\Sigma$, given by the normal component of the Killing field $E_3$ of translations parallel to the $z$-axis. The even solution is given by the \emph{support function} $q_\Sigma\colon\Sigma\to\R$, $q_\Sigma=\langle X,\vec n_\Sigma\rangle$, where $X$ is the position vector. From \eqref{eq:Gaussmap}, we have that
\begin{equation}\label{eq:candidatesols}
\nu_3=\dot x \quad\text{and}\quad q_\Sigma=x\dot z-z\dot x.
\end{equation}
Note that $q_\Sigma$ is a Jacobi field only along minimal surfaces,\footnote{If $X\colon\Sigma\to\R^3$ is a surface in $\R^3$ with constant mean curvature $H_\Sigma$ and support function $q_\Sigma$, then $J_\Sigma q_\Sigma=-2H_\Sigma$. In fact, $q_\Sigma$ can be written as $\tfrac{\dd}{\dd t}\big\vert_{t=1}\langle X_t,\vec n_\Sigma\rangle$, where $X_t=tX$ is a CMC variation of $X$ with mean curvature $H(X_t)=\frac 1tH_\Sigma$.}
and it vanishes along $\partial\Sigma$, since $\Sigma$ is normal.
Thus, the space of odd solutions of \eqref{eq:SturmLiouville} has dimension $1$, as well as the space of even solutions of \eqref{eq:SturmLiouville}. This implies that $S_0$ must be a multiple of either $\nu_3$ or $q_\Sigma$, according to its parity. A direct computation shows that the unit speed parametrization $\gamma(s)=(x(s),z(s))$ of the catenary \eqref{eq:catenoid-generatrix} is
\begin{equation}\label{eq:xzcatenoid}
x(s)=c\sqrt{1+\frac{s^2}{c^2}}, \quad z(s)=c\log\left(\frac sc+\sqrt{1+\frac{s^2}{c^2}} \right),
\end{equation}
and $L_\gamma=2c\sinh(z_c/c)$. It is easy to verify that neither of \eqref{eq:candidatesols} satisfies the boundary conditions \eqref{eq:BVforS}. Thus, there is no nontrivial solution of \eqref{eq:SturmLiouville} with $\lambda=0$ satisfying \eqref{eq:BVforS}, hence the critical catenoid is equivariantly nondegenerate.
\end{proof}

\begin{remark}
The nullity of the critical catenoid and its equivariant nondegeneracy were also obtained by M\'aximo, Nunes, and Smith~\cite{MaxNunSmi13}. 
Very recently, the Jacobi operator \eqref{eq:JacobiSigma} of the critical catenoid was further studied (independently) by Devyver~\cite{devyver}, Smith and Zhou~\cite{smith-zhou}, and Tran~\cite{tran}, who established that its Morse index is equal to $4$ and recomputed its nullity (note \cite{devyver,tran} use a convention for the nullity different from ours).
\end{remark}

\subsection{Delaunay annuli}
Since the critical catenoid is equivariantly nondegenerate, it can be deformed through other CMC annuli (which we later identify as \emph{Delaunay annuli}) varying both the values of its mean curvature and of the ambient curvature.

\begin{proposition}\label{prop:deformationcritcatenoid}
The critical catenoid is the member $\Sigma_{(0,0)}$ of a smooth $2$-param\-eter family $\Sigma_{(h,\lambda)}$ of free boundary annuli with constant mean curvature $h$ inside the unit ball $B^3_\lambda$ of constant curvature $\lambda$, which is unique modulo congruence.
Each $\Sigma_{(h,\lambda)}$ is rotationally symmetric, and invariant with respect to reflection about a plane orthogonal to the rotation axis.
In particular, fixing the axis of rotational symmetry determines $\Sigma_{(h,\lambda)}$ uniquely.
\end{proposition}

\begin{proof}
From Lemma~\ref{lemma:catenoidnondegenerate}, the critical catenoid $\Sigma_0\subset B^3_0$ is equivariantly nondegenerate.
By Lemma~\ref{lemma:killingconstcurv}, every Killing field of $B^3_0$ belongs to a smooth $1$-parameter family of Killing fields of $B^3_\lambda$. Thus, the existence of $\Sigma_{(h,\lambda)}$ and its uniqueness modulo congruence follow from Theorem~\ref{thm:main}.
In order to prove that $\Sigma_{(h,\lambda)}$ have the claimed symmetries, one applies the Implicit Function Theorem on the space of (unparametrized) embeddings into $B^3_\lambda$ that are invariant under
rotations around the $z$-axis and reflections about the $xy$-plane. Indeed, the critical catenoid is nondegenerate in this space, as it does not admit nontrivial rotationally symmetric even Jacobi fields (see Lemma~\ref{lemma:catenoidnondegenerate}).
This yields the existence of a smooth family $\Sigma'_{(h,\lambda)}$ with these symmetries, which must be congruent to $\Sigma_{(h,\lambda)}$ by uniqueness.
\end{proof}

\begin{proposition}\label{prop:foliationcritcatenoid}
Fix a line $\ell$ through the origin in $\R^3$, and let $\Sigma_h$ be the free boundary CMC annulus $\Sigma_{(h,0)}$ in $B^3_0$  from Proposition~\ref{prop:deformationcritcatenoid} which is rotationally symmetric with respect to $\ell$. For $\varepsilon>0$ sufficiently small, $\{\Sigma_h\}_{h\in(-\varepsilon,\varepsilon)}$ determines a foliation of a neighborhood of the critical catenoid $\Sigma_0=\Sigma_{(0,0)}$.
\end{proposition}

\begin{proof}
Without loss of generality we may assume that $\ell$ is the $z$-axis. From Proposition~\ref{prop:deformationcritcatenoid}, $\Sigma_h$ is invariant under reflections about the $xy$-plane. Let $V$ be the variational field along $\Sigma_0$ corresponding to the variation $\Sigma_h$, and set $\psi:=\langle V,\vec n_0\rangle$, where $\vec n_0$ is the unit normal along $\Sigma_0$. Clearly, $\psi$ is a solution to \eqref{eq:Lpsi=1}. From Proposition~\ref{prop:foliatedeqnondeg}, it suffices to show that $\psi$ does not have zeros in $\Sigma_0$.

By the symmetries of $\Sigma_h$, it is easily deduced that $\psi$ is obtained from an even solution $S_0$ to \eqref{eq:eqforSn2} with $n=0$, satisfying the boundary conditions \eqref{eq:BVforS}. In these equations, we consider the data \eqref{eq:xzcatenoid} corresponding to the critical catenoid $\Sigma_0$. Furthermore, up to rescaling, we may assume $c=1$. Direct computations show that solutions of this boundary value problem satisfy
\begin{equation}\label{eq:ODEcatenoid}
-S''-\frac{s}{1+s^2}S'-\frac{2}{(1+s^2)^2}S=1,
\end{equation}
with boundary conditions
\begin{equation}\label{eq:BVforScatenoid}
S'(-s_1)+S(-s_1)=0, \quad \text{ and }\quad S'(s_1)-S(s_1)=0,
\end{equation}
with $s_1:=\sinh z_1$, where $z_1$ is the positive solution to $z_1=\coth z_1$. The general solution to \eqref{eq:ODEcatenoid} satisfying $S(0)=c_1$ and $S'(0)=c_2$ is given by
\begin{equation}\label{eq:ODEgeneralsol}
S(s)=(c_1-\tfrac14s^2)+\frac{s}{\sqrt{1+s^2}}\big(c_2-(c_1+\tfrac14)\sinh^{-1}s\big).
\end{equation}
There is exactly one solution $S_0$ to \eqref{eq:ODEcatenoid} and \eqref{eq:BVforScatenoid}, which is an even function. In order to obtain $S_0$, we set $c_2:=0$ and
\begin{equation}
c_1:=-\frac{1}{4}\left(\frac{z_1^2+1}{z_1^2-1}-\frac{z_1^2}{\left(z_1^2-1\right)^{3/2}}\right)=-\frac14(\cosh 2z_1-\cosh^2 z_1\sinh z_1),
\end{equation}
so that the first condition in \eqref{eq:BVforScatenoid} is satisfied. This suffices to ensure that \eqref{eq:BVforScatenoid} is satisfied, since $S_0$ is even.
Straightforward numeric estimates give $c_1\in (-\tfrac14,0)$, more precisely, $c_1\cong-0.152$. Direct inspection of \eqref{eq:ODEgeneralsol} shows that $S_0(s)<0$ for $s\in [-s_1,s_1]$, hence $\psi$ does not have zeros on $\Sigma_0$, concluding the proof.
\end{proof}

Constant mean curvature surfaces of revolution in $\R^3$ are commonly known as \emph{Delaunay surfaces}, as they were classified by Delaunay~\cite{del} in 1841. Their generatrix was ingeniously observed to be the \emph{roulette} $\gamma$ of some conic section $\beta$, which is the curve traced by one of the foci of $\beta$ as it rolls without slipping along a line $\ell$. Parametrizing the line and the conic section as $\ell\colon\R\to\C$ and $\beta\colon\R\to\C$ respectively, with $\ell(0)=\beta(0)$, $\ell'(0)=\beta'(0)$, and $|\ell'(s)|=|\beta'(s)|=1$, the parametrization $\gamma\colon\R\to\C$ for the corresponding roulette is easily derived to be
\begin{equation}\label{eq:roulette}
\gamma(s)=\ell(s)-\frac{\ell'(s)}{\beta'(s)}\big(\beta(s)-\gamma(0)\big), \quad s\in\R,
\end{equation}
where the starting point $\gamma(0)\in\C$ is to be chosen as one of the foci of $\beta$.

Conic sections $\beta$ depend smoothly on two positive parameters, called \emph{eccentricity} $e$ and \emph{focal parameter} $p$. The conic is a circle, ellipse, parabola, or hyperbola respectively when $e=0$, $e\in (0,1)$, $e=1$, or $e\in(1,+\infty)$, while $p$ is the distance between one of the foci and the directrix, and serves as a rescaling parameter.
\begin{figure}[ht]
\centering
\includegraphics[scale=0.75]{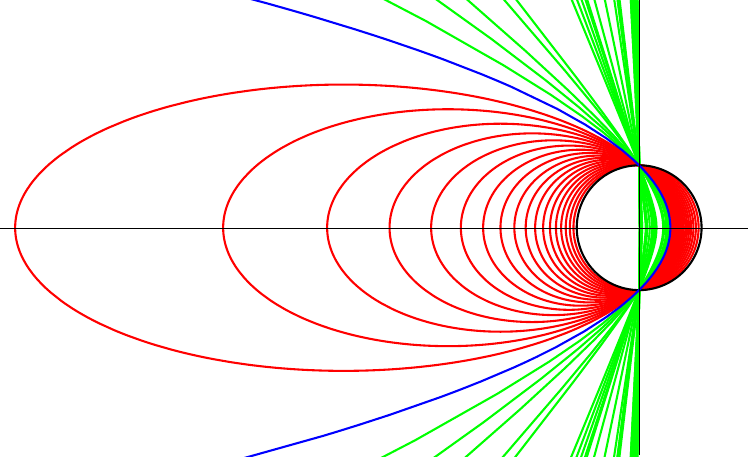}
\caption{The conics \eqref{eq:polarconics} with $e>0$ and $p=1$. The curves in red are ellipses ($0<e<1$), the curve in blue is a parabola ($e=1$) and the curves in green are hyperbolas ($e>1$).}
\label{fig:conics}
\end{figure}
In polar coordinates, a conic with eccentricity $e>0$, focal parameter $p>0$ and one of the foci at the origin is given by
\begin{equation}\label{eq:polarconics}
r(\theta)=\frac{e\,p}{1+e\cos\theta}.
\end{equation}
Reparametrizing $\widetilde\beta(\theta)=r(\theta)\cos\theta+i\,r(\theta)\sin\theta$ by arc length, one obtains the conic $\beta\colon\R\to\C$ and hence the corresponding roulette $\gamma$ via \eqref{eq:roulette}, using $\gamma(0)=0$ and $\ell(s)=\frac{e\,p}{1+e}+i\,s$, see Figure~\ref{fig:roulette}.
\begin{figure}[ht]
\centering
\includegraphics[scale=0.5]{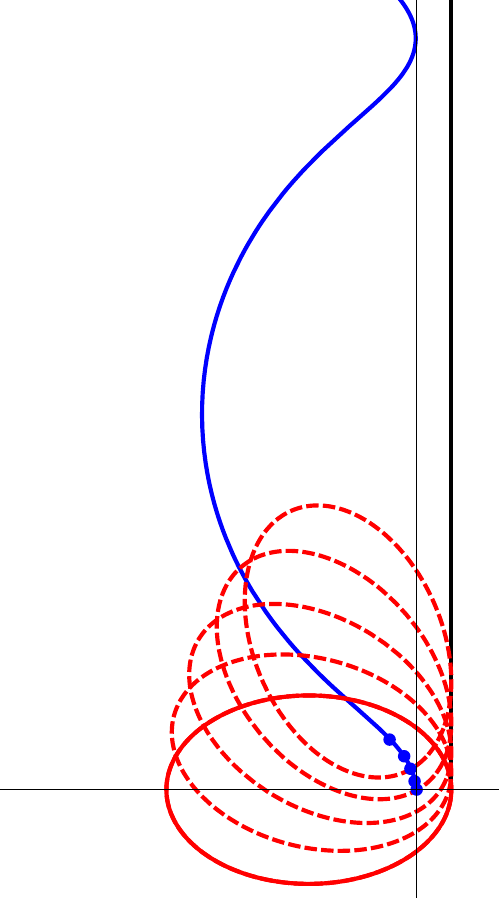}
\begin{pgfpicture}
\pgfputat{\pgfxy(-0.3,4)}{\pgfbox[center,center]{$\ell$}}
\pgfputat{\pgfxy(-2.4,4)}{\pgfbox[center,center]{$\gamma$}}
\pgfputat{\pgfxy(-0.5,0.5)}{\pgfbox[center,center]{$\beta$}}
\end{pgfpicture}
\hspace{2cm}\includegraphics[scale=0.5]{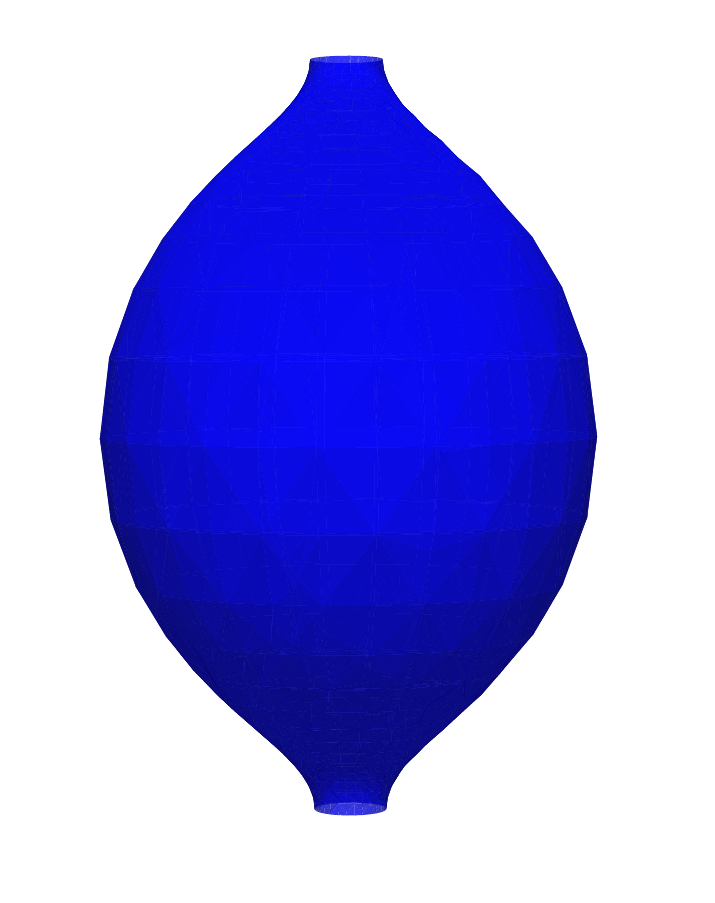}
\caption{Roulette $\gamma$ traced by one of the foci of the conic $\beta$ as it rolls along the line $\ell$, and the corresponding surface of revolution.}
\label{fig:roulette}
\end{figure}
Since $\beta$ and $\ell$ depend smoothly on $p$ and $e$, the roulette $\gamma$ given by \eqref{eq:roulette} and hence the corresponding Delaunay surface $\mathcal D$ in $\R^3$ also depends smoothly on these $2$ parameters.
Thus, Delaunay surfaces are members of a smooth $2$-parameter family $\mathcal D_{(e,p)}$ of CMC surfaces of revolution, obtained by revolving the roulette $\gamma$ of a conic with eccentricity $e$ and focal parameter $p$ around the line $\ell$.\footnote{An explicit parametrization of $\mathcal D_{(e,p)}$ involves elliptic integrals (needed to parametrize $\beta$ by arc length), and can be found, e.g., in \cite{bendito,del,Kenmotsu}.}
The mean curvature of $\mathcal D_{(e,p)}$ can be computed to be
\begin{equation}
H(e,p)=\frac{|e^2-1|}{e\,p}.
\end{equation}
The surface $\mathcal D_{(e,p)}$ is called an \emph{unduloid}, \emph{catenoid}, or \emph{nodoid}, according to the cases $0<e<1$, $e=1$, and $e>1$ respectively, that is, depending on the originating conic $\beta$ being an ellipse, parabola, or hyperbola, see Figure~\ref{fig:delaunay-surfaces}. The parameter $p>0$ corresponds to ambient homotheties, as the rescaled surface $\alpha\,\mathcal D_{(e,p)}$ is congruent to $\mathcal D_{(e,\alpha p)}$. Together with round cylinders and round spheres, which correspond to limits $e=0$ and $e=+\infty$, these are all the CMC surfaces of revolution in $\R^3$.
\begin{figure}[ht]
\centering
\includegraphics[scale=0.42]{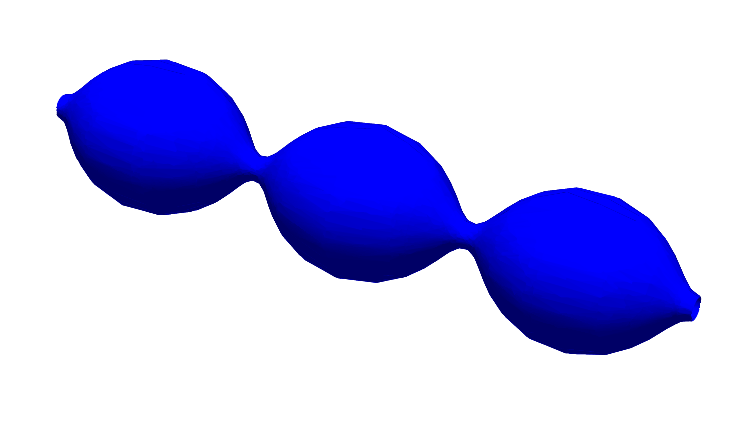}
\includegraphics[scale=0.26]{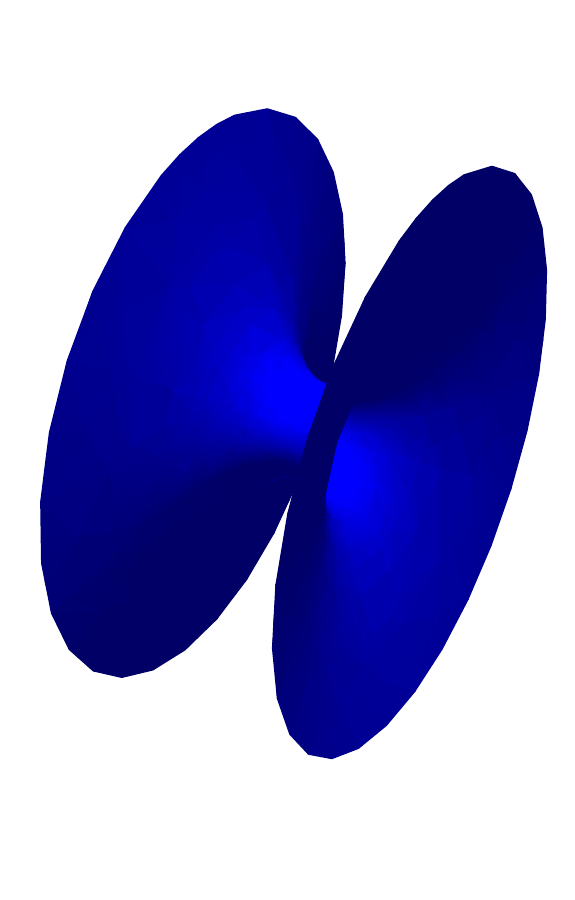}
\includegraphics[scale=0.35]{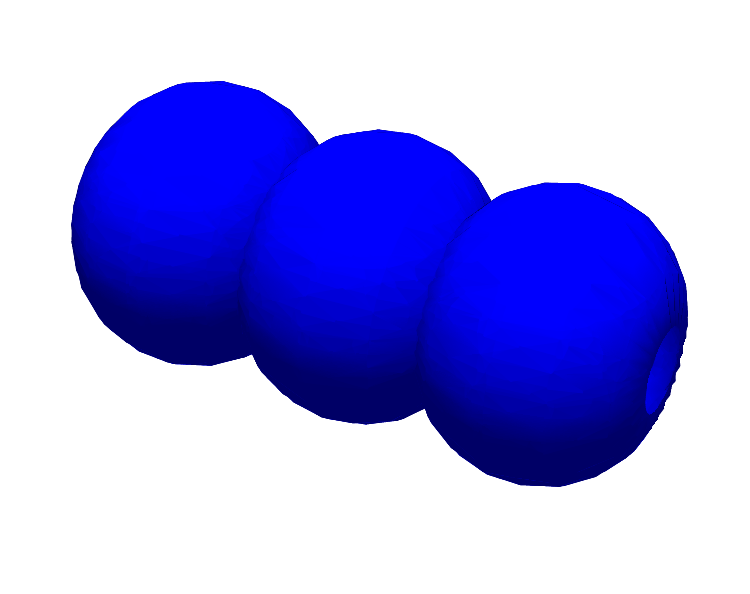}
\includegraphics[scale=0.42]{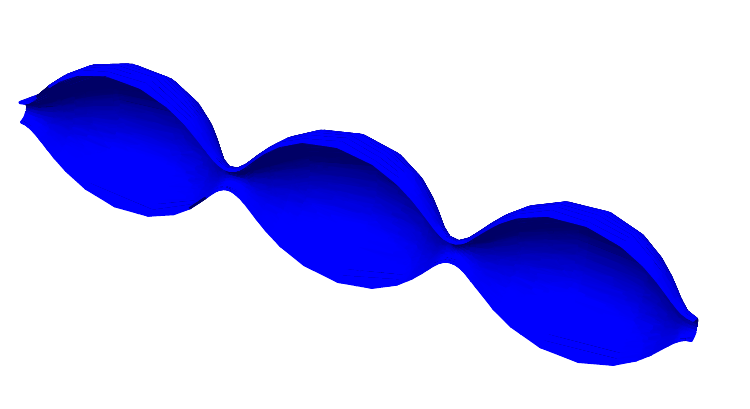}
\includegraphics[scale=0.26]{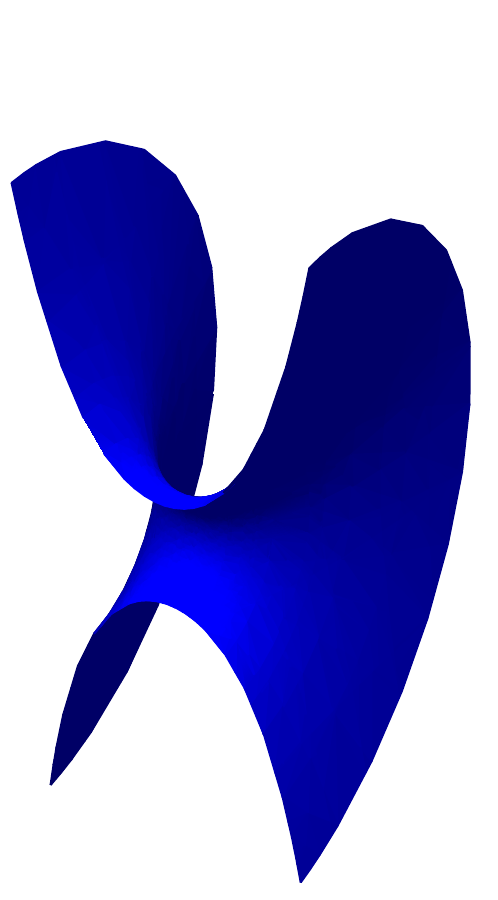}
\includegraphics[scale=0.35]{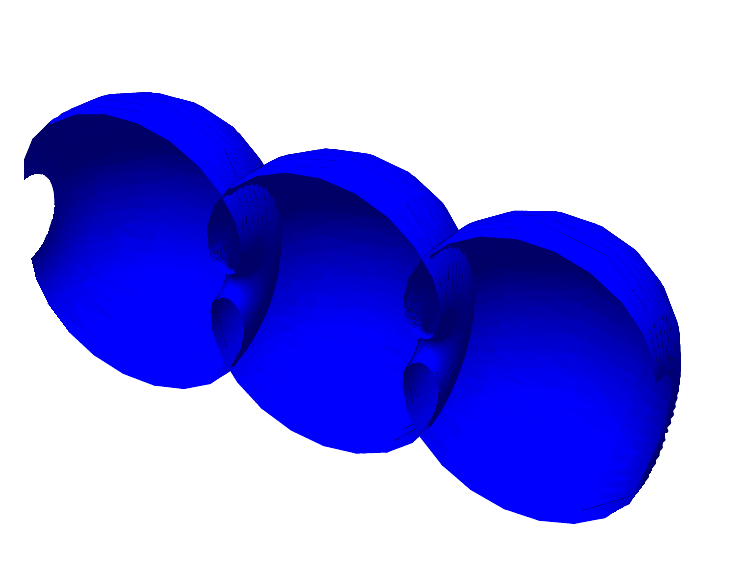}
\caption{Compact portions of an unduloid, a catenoid and a nodoid in $\R^3$, and the corresponding half sections.}
\label{fig:delaunay-surfaces}
\end{figure}

As above, up to ambient isometries, assume that the rotation axis $\ell$ of the Delaunay surface $\mathcal D_{(e,p)}$ is the $z$-axis, and that the roulette $\gamma$ is contained in the $xz$-plane. In this case, the portion of $\mathcal D_{(e,p)}$ contained in the ball $B^3(\varrho)\subset\R^3$ of radius $\varrho$ centered at the origin is a normal surface in $B^3(\varrho)$ if and only if the tangent line to $\gamma$ at the point $\gamma(s_*)$, for which $|\gamma(s_*)|=\varrho$, passes through the origin (cf.~Figure \ref{fig:crit-catenoid}). As $\gamma$ depends smoothly on the eccentricity $e$ and focal parameter $p$ of the corresponding rolling conic, the Implicit Function Theorem can be used to show that there exists a smooth function $\varrho(e,p)$, linear in $p$, such that $\mathcal D_{(e,p)}\cap B^3(\varrho(e,p))$ is a normal surface in $B^3(\varrho(e,p))$. Thus, Delaunay surfaces provide a smooth $1$-parameter family
\begin{equation}
\mathcal A_e:=\mathcal D_{(e,\rho(e,1))}\cap B^3(1)
\end{equation}
of free boundary annuli in the unit ball, with constant mean curvature $h(e):=H(e,\varrho(e,1))$, see also \cite{RosVergasta}. From the above discussion, the surface $\mathcal A_1$ is the critical catenoid. By the uniqueness modulo congruence in Proposition~\ref{prop:deformationcritcatenoid}, it follows that $\Sigma_{(h(e),0)}$ must be congruent to the Delaunay annulus $\mathcal A_e$. Notice that $\{\mathcal A_e\}_{e\in (0,+\infty)}$ form a foliation of $B^3$, cf.~Proposition~\ref{prop:foliationcritcatenoid}.

Delaunay surfaces also exist in the other space forms $\mathds S^3$ and $\mathds H^3$, and can be described as the rotationally symmetric CMC surfaces in such spaces, see \cite{hsiang2,HyndParkMcCuan}. Just as in the Euclidean case, these Delaunay surfaces induce free boundary CMC surfaces in any unit ball with constant curvature, see \cite{souam}. By a uniqueness modulo congruence argument similar to the above, the surfaces in the family $\Sigma_{(h,\lambda)}$ from Proposition~\ref{prop:deformationcritcatenoid} can be identified as Delaunay annuli in $\mathds H^3(1/\sqrt{-\lambda})$ or $\mathds S^3(1/\sqrt \lambda)$, according to $\lambda<0$ or $\lambda>0$.

\begin{remark}\label{rem:moretopology}
In Sections~\ref{sec:disks} and \ref{sec:delaunay}, we studied free boundary CMC disks and annuli in a unit ball of constant curvature as deformations of minimal surfaces in the Euclidean unit ball $B^3$. In principle, our main result (Theorem~\ref{thm:main}) can also be used to produce free boundary CMC surfaces with more complicated topology.

Possible starting points would be the free boundary minimal surfaces of genus~$0$ and any number $r$ of boundary components, recently produced by Fraser and Schoen~\cite{fraser-schoen3}, using extremal metrics for Steklov eigenvalue problems; and by Folha, Pacard, and Zolotareva~\cite{folha} for very large $r$, using gluing techniques. It is not clear whether all of these surfaces are (equivariantly) nondegenerate, but it is natural to expect that many of them are. Furthermore, in higher dimensions, if the new free boundary minimal hypersurfaces obtained by Freidin, Gulian, and McGrath~\cite{fgm} using equivariant methods are equivariantly nondegenerate, then it seems very likely that our deformation result recovers all the free boundary CMC hypersurfaces constructed by Cruz, Palmas, and Reyes~\cite{cruz-palmas-reyes} in the Euclidean unit ball, besides providing new examples of such CMC hypersurfaces in a unit ball with constant curvature.
\end{remark}

\appendix
\section{Proof of Lemma~\ref{thm:rightinverse}}\label{app:whitney}
Let $d,k>0$ be fixed integers.
Given a continuous function $h\colon \R^d \to\R$, the map $F_h\colon\R^{d+1}\to\R$ defined by
\begin{equation}\label{eq:defFg}
F_h(x_1,\ldots,x_d,z)=\frac1{z^{d-1}}\int_{Q(x_1,\ldots,x_d,z)}h(s_1,\ldots,s_d)\,\mathrm ds_1\ldots\mathrm ds_d,
\end{equation}
where $Q(x_1,\ldots,x_d,z)$ is the cube given by the cartesian product $\prod\limits_{i=1}^d\left[x_i-\frac z2,x_i+\frac z2\right]$, satisfies the following properties:
\begin{itemize}
\item it admits a continuous extension to the hyperplane $z=0$ by setting \[F_h(x_1,\ldots,x_d,0)=0;\]
\item it is of class $C^1$ on $\R^{d+1}$;
\item all its partial derivatives have the same regularity as $h$ (in particular, $F_h$ is of class $C^{k+1,\alpha}$ on $\R^{d+1}$ if $h$ is of class $C^{k,\alpha}$);
\item $\displaystyle\frac{\partial F_h}{\partial z}(x_1,\ldots,x_d,0)=h(x_1,\ldots,x_d)$ for all $(x_1,\ldots,x_d)$;\smallskip

\item it depends linearly on $h$.
\end{itemize}
\begin{remark}\label{thm:rempartder}
Using induction, one can show that if $h$ is of class $C^k$, for all $j\le k$,  the partial derivatives of
order $j$ of  $F_h$ are linear combinations of integrals of the partial derivatives of order up to $j-1$ of the function $h$, composed with linear functions in the variables.
\end{remark}
\begin{proposition}
Let $\Sigma$ be a compact Riemannian $n$-manifold with boundary $\partial\Sigma$, and let $\vec\nu$ denote the unit normal vector field along $\partial\Sigma$ pointing inward.\footnote{Observe that here we use the opposite convention for the orientation of the unit normal.} For all $k\ge0$ and $\alpha\in\left(0,1\right]$ there exists a bounded linear map $C^{k,\alpha}(\partial\Sigma)\ni g\mapsto\mathcal F_g\in C^{k+1,\alpha}(\Sigma)$ such that $\mathcal F_g\equiv0$ on $\partial\Sigma$ and $\vec\nu(\mathcal F_g)=g$ for all $g\in C^{k-1,\alpha}(\partial\Sigma)$.
\end{proposition}
\begin{proof}
Choose a finite set of local charts $(U_r,\varphi_r)$, $r=1,\ldots, N$ on $\Sigma$ satisfying the following properties:
\begin{itemize}
\item[(a)] $U_r$ is a connected open subset of $\Sigma$, with $U_r\cap\partial\Sigma\ne\emptyset$ for all $r$;
\item[(b)] $U:=\bigcup_{r=1}^NU_r$ is an open neighborhood of $\partial\Sigma$;
\item[(c)] $\varphi_r$ is a diffeomorphism from $U_r$ onto $\mathds R^{n-1}\times\left[0,+\infty\right)$ carrying $\partial\Sigma\cap U_r$ onto $\mathds R^{n-1}\times\{0\}$;
\item[(d)] $(\mathrm d\varphi_r)_p(\vec\nu_p)=\frac\partial{\partial z}$ for all $p\in\partial\Sigma\cap U_r$.
\end{itemize}
Set $U_0=\Sigma\setminus\partial\Sigma$, so that $(U_r)_{r=0,\ldots,N}$ is an open cover of $\Sigma$, and
let $(\psi_r)_{r=0}^N$ be a smooth partition of unity subordinated to such cover. \smallskip

Given $g\in C^{k,\alpha}(\partial\Sigma)$,
for all $r$, consider the function $h_r=g\circ\varphi_r^{-1}\colon\R^{n-1}\to\R$, which is of class $C^{k,\alpha}$. Let $F_{h_r}\colon\R^{n}\to\R$ be the $C^{k+1,\alpha}$-extension of
$h_r$ defined in formula \eqref{eq:defFg}, and set $\mathcal F_r=F_{g_r}\circ\varphi_r$ be the corresponding $C^{k+1,\alpha}$-map on $U_r$. Finally, define a $C^{k+1,\alpha}$ function on $\Sigma$ as \[\mathcal F_g=\sum_{r=1}^N\psi_r\cdot\mathcal F_r.\]
Let us check that $\mathcal F_g$ has the desired properties. Clearly, $\mathcal F_g$ depends linearly on $g$. The continuity of $\mathcal F$ is obtained from the observation in Remark~\ref{thm:rempartder}. Since all the $F_r$ vanish on $\mathds R^{n-1}\times\{0\}$, by (c) all the $\mathcal F_r$ vanish on $\partial\Sigma$, hence $\mathcal F_g$ vanish on $\partial\Sigma$.
Since $\frac{\partial F_{g_r}}{\partial z}(x_1,\ldots,x_{n-1},0)=g_r(x_1,\ldots,x_{n-1})$, by (d) we get $\vec\nu_p(\mathcal F_r)=g(p)$ for all $p\in\partial\Sigma\cap U_r$. On the other hand, if $p\in\partial\Sigma\setminus U_r$, then $\psi_r(p)=0$. Hence, for $p\in\partial\Sigma$:
\[\vec\nu_p\big(\mathcal F_g\big)=\sum_{r=1}^N\big[\vec\nu_p(\psi_r)\cdot\mathcal F_r(p)+\psi_r(p)\vec\nu_p(\mathcal F_r)\big]=\sum_{r=1}^N\psi_r(p)g(p)=g(p).\]
This concludes the proof.
\end{proof}

\end{document}